\documentclass[11pi, reqno]{article}

\usepackage[english]{babel}

\usepackage[utf8]{inputenc}
\usepackage[T1]{fontenc}
\usepackage{mathtools}   
\usepackage{amsmath}
\usepackage{amssymb}
\usepackage{amsfonts}
\usepackage{amsthm}
\usepackage{graphicx}
\usepackage{geometry}
\usepackage{color}
\usepackage[color,matrix,arrow]{xy}
\usepackage{xy}
\input xy
\xyoption{all}
\usepackage{tikz-cd}
\newtheorem{theorem}{Theorem}[section]
\newtheorem{lemma}[theorem]{Lemma}
\newtheorem{proposition}[theorem]{Proposition}
\newtheorem{corollary}[theorem]{Corollary}

\newtheorem{definition}[theorem]{Definition}
\newtheorem{remark}[theorem]{Remark}

\theoremstyle{definition}

\usepackage{enumerate}
\usepackage{comment}

\begin{document}
\title{The coarse Baum--Connes conjecture for certain relative expanders
\thanks{The second author is partially supported by NSFC (No. 11831006, 12171156) and the Science and Technology Commission of Shanghai Municipality (No. 22DZ2229014). The third author is supported by NSF (No. 1700021, 2000082, 2247313), and the Simons Fellows Program.}
		 }
\author{Jintao Deng, Qin Wang, Guoliang Yu}
\date{}	\maketitle

\begin{abstract}
Let $\left( 1\to N_m\to G_m\to Q_m\to 1 \right)_{m\in \mathbb{N}}$ be a sequence of extensions of finite groups such that their coarse disjoint unions have bounded geometry. In this paper, we show that if the coarse disjoint unions of $\left( N_m \right)_{m\in \mathbb{N}} $ and
$\left( Q_m \right)_{m\in \mathbb{N}} $ are coarsely embeddable into Hilbert space, then the coarse Baum--Connes conjecture holds for the coarse disjoint union of $\left( G_m \right)_{m\in \mathbb{N}}$.

As an application, the coarse Baum--Connes conjecture holds for the relative expanders constructed by G. Arzhantseva and R. Tessera, and the special box spaces of free groups discovered by T. Delabie and A. Khukhro, which do not coarsely embed into Hilbert space, yet do not contain a weakly embedded expander. This enlarges the class of metric spaces known to satisfy the coarse Baum--Connes conjecture. In particular, it solves an open problem raised by G. Arzhantseva and R. Tessera on the coarse Baum--Connes conjecture for relative expanders.
\end{abstract}
\tableofcontents
\section{Introduction}

The coarse Baum--Connes conjecture is a geometric analogue of the Baum--Connes conjecture which has important applications to geometry, topology and analysis. More precisely, it states that the Baum--Connes assembly map
$$\mu: \lim_{d\to \infty} K_*(P_d(X))\to K_*(C^*(X))$$
for a metric space $X$ is an isomorphism, where on the left hand side $\lim\limits_{d\to \infty} K_*(P_d(X))$ is the limit of the locally finite $K$-homology group of the Rips complexes of $X$, while on the right hand side $K_*(C^*(X))$ is the $K$-theory group
of the Roe algebra of $X$.  The left hand side of this conjecture is local and therefore computable, while the right hand side is global and is the receptacle of higher indices of elliptic operators.
A positive answer to this conjecture would provide a complete solution to the problem of computing $K$-theoretic indices for elliptic operators on non-compact spaces. In particular, it implies the Novikov conjecture on homotopy
invariance of higher signatures for closed manifolds when $X$ is a finitely generated group (e.g. the fundamental group of a closed manifold) equipped with a word length metric, and the Gromov's zero-in-the-spectrum conjecture and the positive scalar curvature conjecture when $X$ is a Riemannian manifold. See \cite{Yu2006ICM} for a comprehensive survey for the coarse Baum--Connes conjecture, and
\cite{Braga-Chung-Li 2020,CWY13,Fukaya-Oguni 2012,Fukaya-Oguni 2015,Fukaya-Oguni 2016,Fukaya-Oguni 2020,Gong-Wu-Yu,OOY-09,Willett-Yu-Adv1,Willett-Yu-Book,Yu-98} for recent developments.


The notion of coarse embedding into Hilbert space was introduced by M. Gromov \cite[p. 211]{Grom93} in relation to the Novikov conjecture (1965). A metric space $X$ is said to admit a {\em coarse embedding into Hilbert space} if there exist a map
$f: X\to H$ from $X$ to a Hilbert space $H$, and two maps $\rho_1$ and $\rho_2$ from $[0, \infty)$ to $[0, \infty)$ such that
$\lim_{t\to \infty} \rho_1(t)=\infty$ and
$$\rho_1(d(x, x'))\leq \|f(x)-f(x')\|\leq \rho_2(d(x, x'))$$
for all $x,x'\in X$. The third author proved that the coarse Baum--Connes conjecture holds for all metric spaces with bounded geometry which admit a coarse embedding into Hilbert space \cite{Yu2000}.


\par
The main result of this paper is the following theorem.
\par
\begin{theorem}\label{thm:main result}
Let $\left( 1\to N_m\to G_m\to Q_m\to 1 \right)_{m\in \mathbb{N}}$ be a sequence of extensions of finite groups with uniformly finite generating subsets. If the coarse disjoint union of $\left( N_m \right)_{m\in \mathbb{N}} $ with the induced metric from the word metric of
$\left( G_m \right)_{m\in \mathbb{N}} $ and the coarse disjoint union of $\left( Q_m \right)_{m\in \mathbb{N}} $ with the quotient metric are
coarsely embeddable into Hilbert space, then the coarse Baum--Connes conjecture holds for the coarse disjoint union of $\left( G_m \right)_{m\in \mathbb{N}}$.
\end{theorem}
\par
We will say that a sequence of group extensions $\left( 1\to N_m\to G_m\to Q_m\to 1 \right)_{m\in \mathbb{N}}$ is a
{\em ``CE-by-CE'' extension}, or has the {\em ``CE-by-CE'' structure}, if it satisfies the assumptions in the above theorem. Similarly, we may talk about ``A-by-A'', ``A-by-CE'', ``CE-by-A'', and so on for obvious meanings, where ``A'' stands for property A introduced by the third author \cite{Yu2000}, which implies ``CE''.
\par
For a long time, weakly embedded expanders
(see Definition \ref{weakly embedded expander} below) were the only known obstruction for a metric space with bounded geometry to coarsely embed into Hilbert space. In \cite{Arzhantseva-Tessera 2015}, however, G. Arzhantseva and R. Tessera introduced a notion of
{\em relative expander} (see Definition \ref{relative expander} below) to construct the first sequence of finite Cayley graphs which does not coarsely embed into any $L^p$-space for
any $1\leq p< \infty$, nor into any uniformly convex Banach space, and yet does not admit any weakly embedded expander.
 In \cite{Delabie-Khukhro}, T. Delabie and A. Khukhro
construct a certain box space of a free group to answer in the affirmative an open problem asked in \cite{Arzhantseva-Tessera 2015}: does there exist a sequence of finite graphs with bounded degree and large girth which does not coarsely embed into Hilbert space and yet does not contain a weakly embedded expander? We observe that all these examples by G. Arzhantseva and R. Tessera, and T. Delabie and A. Khukhro, are sequences of extensions of finite groups which have the ``A-by-CE'' structure. Consequently, we have the following result, which in particular solves an open problem raised by G. Arzhantseva and R. Tessera in \cite[Section 8, Open Problems]{Arzhantseva-Tessera 2015}.
\par
\begin{corollary}
The coarse Baum--Connes conjecture holds for all relative expanders exhibited by G. Arzhantseva and R. Tessera, and the
special box spaces of free groups discovered by T. Delabie and A. Khukhro, which do not coarsely embed into Hilbert space, yet do not contain
a weakly embedded expander.
\end{corollary}
\par

\par
The basic strategy of the proof of Theorem \ref{thm:main result} is to apply the index-theoretic localization technique on finite dimensional Hilbert spaces detailed in the book \cite[Chapter 12]{Willett-Yu-Book}, which is based on a result of Yu in \cite{Yu2000}, to the case of sequences of extensions of finite groups. We use the coarse embedding of the quotients $(Q_m)_{m\in \mathbb{N}}$ into a Hilbert space to
construct the twisted Roe algebra and the twisted localization algebra of $(G_m)_{m\in \mathbb{N}}$, and formulate the twisted coarse Baum--Connes conjecture. The cutting-and-pasting technique enables us to decompose these twisted algebras, so as to reduce the twisted coarse Baum--Connes conjecture for
$(G_m)_{m\in \mathbb{N}}$ to the coarse Baum--Connes conjecture with  coefficients for $(N_m)_{m\in \mathbb{N}}$, which in turn can be proved by  the coarse embeddability of the sequence $(N_m)_{m \in \mathbb{N}}$ into Hilbert space. Finally, the coarse Baum--Connes conjecture for $(G_m)_{m \in \mathbb{N}}$ follows from the twisted coarse Baum--Connes conjecture by the fact that the index maps from the $K$-theory of Roe algebras and localization algebras to the $K$-theory of the twisted Roe algebras and the twisted localization algebras are isomorphic.
\par

\par
The paper is organized as follows. In Section 2, we briefly review all the examples of relative expanders constructed
by G. Arzhantseva and R. Tessera, and a special box space of free group discovered by T. Delabie and A. Khukhro, which do not coarsely embed into Hilbert spaces and yet do not contain a weakly embedded expander. We indicate that all these examples have the "A-by-CE" structure. In Section 3, we recall the concepts of the Roe algebras, localization algebras, and the coarse Baum--Connes conjecture. In Section 4, we first introduce the twisted Roe algebras and twisted localization algebras for the sequence of extension groups with coefficients coming from the coarse embedding of the quotient groups into Hilbert space, then we introduce the index maps which are used to reduce coarse Baum--Connes conjecture to the twisted coarse Baum--Connes conjecture. In Section 5, we prove the twisted coarse Baum--Connes conjecture for the sequence of extensions. In Section 6, we complete the proof of the main result.


\section{A-by-CE group extension structure of relative expanders and box spaces}
In this section, we briefly review the recent discoveries of relative expanders due to G. Arzhantseva and R. Tessera \cite{Arzhantseva-Tessera 2015}, and certain box spaces of free groups due to T. Delabie and A. Khukhro \cite{Delabie-Khukhro}, which do not coarsely embed into Hilbert space yet contain no weakly embedded expanders. We observe that all these examples have the ``A-by-CE'' structure as sequences of finite group extensions.
\par
For a finite connected graph $X$ with $|X|$ vertices and a subset $A\subset X$, denote by $\partial A$ the set of edges between $A$ and
$X\backslash A$. The {\em Cheeger constant} of $X$ is defined as
$$h(X):= \min_{1\leq |A| \leq |X|/2} \frac{|\partial A|}{|A|}.$$
An {\em expander} is a sequence $(X_m)_{m\in \mathbb{N}}$ of finite connected graphs with uniformly bounded degree,
such that $|X_m|\to \infty$, and $h(X_m)\geq c$ uniformly over $m\in \mathbb{N}$ for some constant $c>0$.
\par
\begin{definition}[weakly embedded expander]\label{weakly embedded expander}
Let $(X_m)_{m\in \mathbb{N}}$ be an expander and let $Y$ be a discrete metric space with bounded geometry. A sequence of maps
$f_m: X_m\to Y$ is a {\em weak embedding of $(X_m)_{m\in \mathbb{N}}$ into $Y$} if there exists $D>0$ such that all $f_m$ are $D$-Lipschitz and
$$\lim_{m\to \infty} \; \sup_{x\in X_m} \; \frac{|f_m^{-1}(f_m(x))|}{|X_m|}=0.$$
\end{definition}
\par
Let us recall the concepts of semidirect product and wreath product of groups.
Let $H$ be a group,  let $Q$ be a group of automorphisms of $H$,
and let $K$ be a group such that there is a surjective homomorphism $\phi: K\twoheadrightarrow Q$. The {\em restricted semi-direct product}
$H\rtimes_Q K$ is the product $H\times K$ with the multiplication rule
$$(h_1, k_1)(h_2, k_2)=\big( h_1 \cdot \phi(k_1) h_2, \, k_1k_2 \big).$$
Note that a semi-direct product is a split extension. Let $H$ and $K$ be finitely generated groups,  and let $\phi: K\twoheadrightarrow Q$ be a surjective homomorphism from $K$ to a
countable discrete group $Q$.
The {\em restricted permutational wreath product of $H$ by $K$ through $Q$} is the semi-direct product
$$H\wr_Q K:= \left( \bigoplus_{Q}H \right) \rtimes K,$$
where $\bigoplus_Q H$ is the group of finitely supported functions $\xi: Q\to H$ with the pointwise multiplication, and
$K$ acts on $\bigoplus_Q H$ by permuting the indices by multiplications on the left via $\phi: K\twoheadrightarrow Q$.
\par

A well known obstruction for a metric space to coarsely embed into a Hilbert space is to admit a weakly embedded expander \cite{Grom03}.
A long standing open problem is: does a weakly embedded expander is the only obstruction to coarse embeddability into Hilbert space?
In \cite{Arzhantseva-Tessera 2015}, G. Arzhantseva and R. Tessera gave first examples of sequences of finite Cayley graphs
of uniformly bounded degree which do not coarsely embed into a Hilbert space but do not contain any weakly embedded expander.
\par
\begin{theorem}[\cite{Arzhantseva-Tessera 2015} Theorem 1 or 7.1]
There exist a finitely generated residually finite group $G$ and a box space $(Y_m)_{m\in \mathbb{N}}$ of $G$ which does not coarsely embed into any
$L^p$ space for $p\in [1.\infty)$, neither into any uniformly curved Banach space, and yet does not admit any sequence of weakly embedded expanders.
\end{theorem}
\par
This result relies on two major observations. The first one says that there are no expanders weakly contained in  "CE-by-CE" group extensions (Proposition 2 in \cite{Arzhantseva-Tessera 2015}).
The second observation is a refined strengthening to an early observation of J. Roe \cite{Roe-03} that relative property (T), as opposite to the Haagerup property, leads to non-embeddability into Hilbert space. This is done by introducing a notion of {\em relative expander} in terms of relative Poincar\'{e} inequality \cite{Arzhantseva-Tessera 2015}.
\par
\begin{definition}[Relative expanders \cite{Arzhantseva-Tessera 2015}, 1.4]\label{relative expander}
Let $\big( G_m \big)_{m\in \mathbb{N}}$  be a sequence of finite groups with generating subsets $S_m$ such that $\sup_m |S_m|<\infty$, and let
$Y_m\subseteq G_m$ be an {unbounded sequence} of subsets of $G_m$, i.e. for any $R>0$ there exists $m$ such that
$Y_m$ is not contained entirely in the ball of radius $R$ around the identity element in $G_m$. Then the sequence of Cayley graphs
$(G_m, S_m)_{m\in \mathbb{N}}$ is said to be a {\em relative expander with respect to $(Y_m)$} if it satisfies the
"relative Poincar\'{e} inequality": there exists $C>0$ such that for every $m\in \mathbb{N}$, for every function
$f: G_m\to \mathcal{H}$  from $G_m$ to a Hilbert space $\mathcal{H}$, and for every $y\in Y_m$, one has
$$\sum_{g\in G_m} \| f(gy)-f(g)\|^2 \leq C \sum_{g\in G_m, s\in S_m} \|f(gs)-f(g)\|^2.$$
\end{definition}
\par
\begin{proposition}(\cite{Arzhantseva-Tessera 2015}, Proposition 3 and Corollary 1.1) \quad
Let $G$ be a finitely generated residually finite group with a finite generating subset $S$. If $G$ has relative property (T) with respect to an {\em infinite} subset $Y\subset G$, then for any filtration $(N_m)$ of $G$, with the quotient maps $\pi_m: G\to G/N_m$, the sequence of Cayley graphs $(G/N_m, \pi_m(S))$ is a relative expander with respect to the unbounded sequence $(\pi_m(Y))_m$. In particular, the box space
$\bigsqcup_n G/N_m$ does not coarsely embed into a Hilbert space.
\end{proposition}
\par
In \cite{Arzhantseva-Tessera 2015}, G. Arzhantseva and R. Tessera provide three examples of relative expanders which are box spaces of certain
carefully defined semi-direct product groups. We observe that all these examples are "A-by-CE" sequence of extensions of finite groups.
\par
\mbox{}\noindent{\bf Example 1: a box space of $ \mathbb{Z}^2 \rtimes_Q \mathbb{F}_3 $}.
Consider the semi-direct product
$$G:= \mathbb{Z}^2 \rtimes_Q \mathbb{F}_3, $$
where $Q$ is the kernel of the surjection $\mathrm{SL}(2, \mathbb{Z})\twoheadrightarrow \mathrm{SL}(2, \mathbb{Z}_2)$, which is generated by $3$ elements.
Take a surjective homomorphism $\pi: \mathbb{F}_3\twoheadrightarrow Q$ mapping the standard generating subset $\{a, b, c\}$
of $\mathbb{F}_3$ onto some generating subset $\{\alpha, \beta, \gamma\}$  of $Q$, so that $\mathbb{F}_3$ acts on $\mathbb{Z}^2$ via the standard action of $Q$ on $\mathbb{Z}^2$.
Then the set $S:=\{(1, 0), (0, 1), a, b, c\}$ is a finite generating subset of $G$.
Define a sequence of semi-direct products (i.e. split extensions) of finite groups as follows:
$$G_m:= \Big( \mathbb{Z}_{2^m} \Big)^2 \rtimes_{Q_m} \Big( \mathbb{F}_3/\Gamma_{3m-3}(\mathbb{F}_3) \Big) $$
where, for all $m\geq 1$, $Q_m<\mathrm{SL}(2, \mathbb{Z}_{2^m})$ is the image of $Q<\mathrm{SL}(2, \mathbb{Z})$ under the quotient map
        $$\pi_{2^m}:\; \mathrm{SL}(2, \mathbb{Z})\twoheadrightarrow \mathrm{SL}(2, \mathbb{Z}_{2^m}),$$
and $\{\mathbb{F}_3/\Gamma_{3m-3}(\mathbb{F}_3)\}_{m\in \mathbb{N}}$ is a specific box space of the free group $\mathbb{F}_3$ constructed by G. Arzhantseva, E. Guentner, and J. \v{S}pakula \cite{AGS-2012}, and generalized by A. Khukhro \cite{Khukhro2014}, which coarsely embeds into a Hilbert space.
It is shown in \cite{Arzhantseva-Tessera 2015} that the sequence  of Cayley graphs $(G_m, S_m)_m$ is a box space of $G$ and a relative expander with respect to the unbounded sequence of subsets $\big( \mathbb{Z}_{2^m} \big)^2 \subset G_m$, $m\in \mathbb{N}$. Since $\mathbb{Z}^2$ is amenable, it is obvious that the sequence $(G_m)_{m\in \mathbb{N}}$ has "A-by-CE" structure.

\par
\mbox{}\noindent{\bf Example 2: a box space of $ \mathbb{Z} \wr_Q \mathbb{F}_U $}.
Consider the generalized wreath product
$$G:=\mathbb{Z}\wr_Q \mathbb{F}_U := \left( \bigoplus_{Q} \mathbb{Z} \right) \rtimes \mathbb{F}_U,$$
where
$Q:=\ker \big( \mathrm{SL}(3, \mathbb{Z}) \twoheadrightarrow \mathrm{SL}(3, \mathbb{Z}_2) \big)$, equipped with a finite generating subset $U$. Define a sequence of generalized wreath products of finite groups as follows:
$$G_m:=  \big( \mathbb{Z}/n\mathbb{Z} \big) \wr_{Q_m} \Big( \mathbb{F}_U/\Gamma_{8m-8}(\mathbb{F}_U) \Big)
            := \left( \bigoplus_{Q_m} \big( \mathbb{Z}/m\mathbb{Z} \big) \right)  \rtimes \Big( \mathbb{F}_U/\Gamma_{8m-8}(\mathbb{F}_U) \Big) $$
where, for all $m\geq 1$, $Q_m<\mathrm{SL}(3, \mathbb{Z}_{2^m})$ is the image of $Q<\mathrm{SL}(3, \mathbb{Z})$ under the quotient map
        $$\pi_{2^m}:\; \mathrm{SL}(3, \mathbb{Z})\twoheadrightarrow \mathrm{SL}(3, \mathbb{Z}_{2^m}),$$
and $\{\mathbb{F}_U/\Gamma_{8m-8}(\mathbb{F}_U)\}_{m\in \mathbb{N}}$ is an embeddable box space of G. Arzhantseva, E. Guentner, and J. \v{S}pakula \cite{AGS-2012}, generalized by A. Khukhro \cite{Khukhro2014}.
It is shown in \cite[Theorem 7.3]{Arzhantseva-Tessera 2015} that the sequence  of Cayley graphs $(G_m, S_m)_m$ is a box space of $G$ and a relative expander with respect to an unbounded sequence of subsets of $(G_m)_{m\in \mathbb{N}}$. Since $ \bigoplus_{Q} \mathbb{Z} $ is an amenable group, it is obvious that the sequence $(G_m)_{m\in \mathbb{N}}$ has "A-by-CE" structure.

\mbox{}\noindent{\bf Example 3: a box space of $\mathbb{Z}\wr_Q \mathbb{F}_3 $}.
Consider the generalized wreath product
$$G:=\mathbb{Z}\wr_Q \mathbb{F}_3 := \left( \bigoplus_{Q} \mathbb{Z} \right) \rtimes \mathbb{F}_3,$$
where $Q:=\ker \big( \mathrm{SL}(2, \mathbb{Z}) \twoheadrightarrow \mathrm{SL}(2, \mathbb{Z}_2) \big)$, equipped with a generating subset
$U=\{\alpha, \beta, \gamma\}$ as in Example 1.
Define a sequence of generalized wreath products of finite groups as follows:
$$G_m:=  \big( \mathbb{Z}/m\mathbb{Z} \big) \wr_{Q_m} \Big( \mathbb{F}_3/\Gamma_{3m-3}(\mathbb{F}_3) \Big)
            := \left( \bigoplus_{Q_m} \big(\mathbb{Z}/m\mathbb{Z} \big) \right)  \rtimes \Big( \mathbb{F}_3/\Gamma_{3m-3}(\mathbb{F}_3) \Big), $$
where $Q_m$ are as in Example 1. It is shown in \cite[Section 7.2]{Arzhantseva-Tessera 2015} that
the sequence  of Cayley graphs $(G_m, S_m)_m$ is a box space of $G$ and a relative expander with respect to an unbounded sequence of subsets of $(G_m)_{m\in \mathbb{N}}$. Since $ \bigoplus_{Q} \mathbb{Z} $ is an amenable group, it is obvious that the sequence $(G_m)_{m\in \mathbb{N}}$ has "A-by-CE" structure.

\mbox{}\noindent{\bf Box spaces of free groups \`{a} la T. Delabie and A. Khukhuro}. In \cite{Delabie-Khukhro} T. Delabie and A. Khukhro answer an open question asked by G. Arzhantseva and R. Tessera in
\cite[Section 8: Open Problems]{Arzhantseva-Tessera 2015}: Does there exist a sequence of finite graphs with bounded degree and large girth
that does not coarsely embed into a Hilbert space and yet has no weakly embedded expander (in particular, a box space of the free group $\mathbb{F}_m$)?
\par
\begin{theorem}[\cite{Delabie-Khukhro}]
There exists a filtration of the free group $\mathbb{F}_3$ such that the corresponding box space does not coarsely embed into a Hilbert space,
but does not admit a weakly embedded expander.
\end{theorem}
\par
The overall structure of the proof is as follows. They first construct a particular sequence of nested finite index normal subgroups $\{N_i\}$ of $\mathbb{F}_3$ with trial intersection so that the corresponding box space is an expander, and consider the sequence of $q$-homology covers, where
$q$ is a certain prime number,  of the quotients $\{ \mathbb{F}_3/N_i\}$: this gives rise to another sequence of subgroups $\Theta(N_i) < N_i$ of $\mathbb{F}_3$ such that the corresponding box space $\Box_{\Theta(N_i)} \mathbb{F}_3$ coarsely embeds into a Hilbert space, following a previous work of A. Khukhro \cite{Khukhro2014}.
They find an ingenuous way to choose a subsequence of quotients
$\Big\{ \mathbb{F}_3/\big( N_{n_i} \cap \Theta(N_{k_i}) \Big\}$
which lie on some path that moves sufficiently slowly away the horizontal expander sequence in the "triangle" of intersections.
They show \cite[Corollary 4.4]{Delabie-Khukhro} that there exist increasing sequences $k_i$ and $n_i$ such that the box space
$\Box_{N_{n_i}\cap \Theta(N_{k_i})} \mathbb{F}_3$
contains a generalized expander and therefore does not coarsely embed into Hilbert space, by a characterization of R. Tessera \cite{Tessera}.
\par
On the other hand,  T. Delabie and A. Khukhro have essentially proved the following general fact \cite[Proposition 2.4]{Delabie-Khukhro}: Let $G$ be a finitely generated residually finite group with two filtrations $\{N_i\}$ and $\{M_i\}$ such that $M_i<N_i$ for all $i\in \mathbb{N}$.
Consider the sequence of group extensions
$$\Big( 1\longrightarrow N_i/M_i \longrightarrow G/M_i \longrightarrow G/N_i \longrightarrow 1 \Big)_{i\in \mathbb{N}} ,$$
where $G/M_i$ and $G/N_i$ are considered with the metric induced by the restriction of the respective box space metrics, and $N_i/M_i$ is considered with the metric induced by viewing $N_i/M_i$ as a subspace of $G/M_i$. Then the quotient maps
$$G\twoheadrightarrow G/M_n; \quad\quad G\twoheadrightarrow G/N_n$$
are asymptotically faithful, so that the quotient maps
$$\pi_n: G/M_n \twoheadrightarrow G/N_n$$
is asymptotically faithful as well. It follows that there exists a sequence $r_n\to \infty$ such that any two points of $N_n/M_n$ are at distance at least $r_n$ from each other. In other words, the asymptotic dimension of the coarse disjoint union $\bigsqcup_i N_i/M_i$ is $0$, which implies property A. Thus, if $\{G/N_i\}_i$ coarsely embeds into a Hilbert space, then the sequence $\{G/M_i\}_i$ has "A-by-CE" structure, so that it does not contain weakly embedded expanders. Applied to the magic triangle in \cite{Delabie-Khukhro}, by our main Theorem \ref{thm:main result} in this paper,
we actually have the following
\par
\begin{corollary}
For {\em all} increasing sequences $k_i$ and $n_i$ of $\mathbb{N}$, the corresponding sequence of extensions of finite groups in the
Delabie--Khukhro magic triangle:
$$\Big( 1\longrightarrow \Theta(N_{k_i})/\big( N_{n_i} \cap \Theta(N_{k_i}) \big)
            \longrightarrow  \mathbb{F}_3/\big( N_{n_i} \cap \Theta(N_{k_i}) \big)
         \longrightarrow \mathbb{F}_3/\Theta(N_{k_i}) \longrightarrow 1 \Big)_{i\in \mathbb{N}} $$
has "A-by-CE" structure. Therefore, the coarse Baum--Connes conjecture holds for the box space
$$\Box_{N_{n_i}\cap \Theta(N_{k_i})} \mathbb{F}_3  \, := \,  \bigsqcup_{i\in \mathbb{N}} \mathbb{F}_3/\big( N_{n_i} \cap \Theta(N_{k_i}) \big) $$
for {\em all} increasing sequences $k_i$ and $n_i$ of $\mathbb{N}$. In particular, the coarse Baum-Connes conjecture holds for the specific box space $\Big\{ \mathbb{F}_3/\big( N_{n_i} \cap \Theta(N_{k_i}) \Big\}_i$ of $\mathbb{F}_3$ by T. Delabie and A. Khukhro in \cite{Delabie-Khukhro}, which does not coarsely embed into Hilbert space yet does not contain a weakly embedded expander.
\end{corollary}


\section{The coarse Baum--Connes conjecture}
In this section, we shall recall the concepts of Roe algebras, localization algebras and the coarse Baum--Connes conjecture.

\subsection{The Roe algebras and the coarse assembly map}


Recall that a metric space is proper if every closed bounded subset is compact. Let $Z$ be a proper metric space. An ample $Z$-module is a separable Hilbert space $H_Z$ equipped with a faithful and non-degenerate $*$-representation of $C_0(Z)$ whose range contains no nonzero compact operators, where $C_0(Z)$ is the algebra of all complex-valued continuous functions on $Z$ vanishing at infinity.

\begin{definition}
Let $Z$ be a proper metric space and $H_{Z}$ an ample $Z$-module.
\begin{enumerate}[(1)]
	\item The support of a bounded linear operator $T: H_{Z} \to H_{Z}$ is defined to be the set of points $(x, y) \in Z\times Z$ for which there exist $f, g \in C_0(Z)$ with $f(x)\neq 0$ and $g(y) \neq 0$ such that $fTg \neq 0$. The support of $T$ is denoted by $\mbox{Supp}(T)$.
	\item For a bounded linear operator $T: H_{Z} \to H_{Z}$, the $Z$-propagation of $T$ is defined by
	$${\rm Prop}_{Z}(T):=\sup\left\{d(x, y): (x, y)\in \mbox{Supp}(T)\right\}.$$
	An operator $T$ is said to have finite $Z$-propagation (or finite propagation) if ${\rm Prop}_{Z}(T)<\infty$.
	\item An operator $T$ is locally compact if the operators $fT$ and $Tf$ are compact operators on $H_{Z}$ for all $f \in C_0(Z)$.
\end{enumerate}
\end{definition}

In the case where $Z$ is a countable discrete metric space, we can choose a specific ample $Z$-module as follows. Let $H_0$ be a separable infinite-dimensional Hilbert space. Then the Hilbert space $\ell^2(Z)\otimes H_0$ is an ample $Z$-module via the representation given by the formula
$$f\cdot (\xi \otimes v)=(f\xi) \otimes v$$
for all $f \in C_0(Z)$, $\xi\in \ell^2(Z)$ and $v \in H_0$.

Since $\ell^2(Z)\otimes H_0=\bigoplus_{x \in Z}\mathbb{C}\cdot\delta_x\otimes H_0$, we can express each bounded operator $T \in B(\ell^2(Z)\otimes H_0)$ as a $Z$-by-$Z$ matrix,
$$T=\left(T_{x, y}\right)_{x,y\in Z}$$
where $T_{x,y}$ is a bounded linear operator on $H_0$. We have that
$${\rm Prop}_{Z}(T)=\sup\left\{d(x,y):T_{x,y}\neq 0\right\}.$$
If $T$ is locally compact, then $T_{x,y}$ is a compact operator on $H_0$ for all $x, y \in Z$.

\begin{definition}\cite{Roe-93}
Let $Z$ be a proper metric space and $H_{Z}$ an ample $Z$-module.
\begin{enumerate}[(1)]
    \item The algebraic Roe algebra, denoted by $\mathbb{C}[Z;H_{Z}]$, is the $*$-algebra of all locally compact, finite propagation operators on $H_{Z}$.
    \item The Roe algebra, denoted by $C^*(Z; H_{Z})$, is defined to be the completion of $\mathbb{C}[Z; H_{Z}]$ under the operator norm on $H_{Z}$.
\end{enumerate}
\end{definition}
Let $Z_1$ and $Z_2$ be proper metric spaces. A map $f:Z_1\rightarrow Z_2$ is said to be a coarse embedding if there exist two non-decreasing functions $\rho_1,\rho_2:[0,\infty)\rightarrow [0,\infty)$ such that
\begin{enumerate}[(1)]
	\item $\lim\limits_{t\to \infty}\rho_i(t)=\infty$ for $i=1,2$;
	\item $\rho_1(d(x,y))\leq d(f(x),f(y))\leq \rho_2(d(x,y))$, for all $x,y\in Z_1$.
\end{enumerate}
A metric space $Z_1$ is said to be coarsely equivalent to another metric space $Z_2$, if there exist a coarse embedding $f:Z_1\rightarrow Z_2$ such that $Z_2$ is equal to the $C$-neighborhood of the image $f(Z_1)$ for some $C>0$.


Let $f: Z_1\rightarrow Z_2$ be a coarse embedding. Let $\{U_i\}_{i\in I}$ be a Borel cover of $Z_2$ with the following properties:
\begin{enumerate}[(1)]
	\item  $\{U_i\}_{i\in I}$ is mutually disjoint;
	\item  each $U_i$ has non-empty interior;
	\item  for any compact $K\subseteq Z_{2}$, the set $\left\{i\in I|U_i\cap K=\emptyset\right\}$ is finite;
	\item  the diameter of $U_i$ is uniformly bounded over $i\in I$.
\end{enumerate}
    Note that the $*$-representation of $C_0(Z_2)$ on $H_{Z_2}$ extends to a $*$-representation of the algebra of all bounded Borel functions on $Z_2$. Since the range of $C_0(Z_2)$ contains no non-zero compact operators,  $\chi_{U_i}H_{Z_2}$ is an infinite dimensional subspace of $H_{Z_2}$ for each $i$, where $\chi_{U_i}$ is the projection corresponding to the characteristic function on $U_i$.
Hence, for each $i\in I$, there always exists an isometry $$V_i:\chi_{f^{-1}(U_i)}H_{\Gamma_{1}} \rightarrow \chi_{U_i}H_{Z_2}.$$
Consequently, we obtain an isometry $$V=\mathop\oplus\limits_{i\in I}V_i: H_{Z_1}=\oplus_i\chi_{f^{-1}(U_i)}H_{Z_1} \rightarrow H_{Z_2}=\oplus_i\chi_{U_i}H_{Z_2},$$
and a $*$-homomorphism
$$
\mbox{Ad}V: \mathbb{C}[Z_1;H_{Z_1}] \to \mathbb{C}[Z_2;H_{Z_2}]
$$
given by $T\mapsto VTV^*$.

When $f:Z\rightarrow Z$ is the identity map, the  isometry $V$ can be chosen as a unitary. It follows that  the definition of $\mathbb{C}[Z; H_Z]$ does not depend on the choice of $H_{Z}$ up to $*$-isomorphisms. Thus, we also denote by $\mathbb{C}[Z]$ and $C^*(Z)$ the algebraic Roe algebra and Roe algebra of a proper metric space $Z$, respectively.

Accordingly, we obtain a homomorphism
$$\mbox{Ad}V_*: K_*(C^*(Z_1))\to K_*(C^*(Z_2))$$
on $K$-theory induced by $\mbox{Ad}V$, where $V:H_{Z_1} \to H_{Z_2}$ is induced by the coarse embedding $f: Z_1 \to Z_2$. The map induced by $\mbox{Ad}V$ on $K$-theory does not depend on the choice of $V$. As above, if $f$ is a coarse equivalence, the isometry $V$ can be chosen as a unitary.

\begin{remark}\label{Rem_map_indu_by coar_map}
Let $Z_1$ and $Z_2$ be countable discrete metric spaces and $f: Z_1 \rightarrow Z_2$ an injective coarse map. Then we have an explicit construction of the homomorphism between the algebraic Roe algebras:
$$\mathbb{C}[Z_1] \rightarrow \mathbb{C}[Z_2]
$$
by
\[
  \left(\mbox{Ad}V(T)\right)_{y,y'} =
  \begin{cases}
        T_{x,x'}, & \text{\mbox{if} $y=f(x)$, $y'=f(y)$;} \\
        0, & \text{\mbox{otherwise}.}
  \end{cases}
\]

\end{remark}

Let us briefly recall the definition of analytic $K$-homology groups. More details can also be found in \cite{Willett-Yu-Book}.
Let $Z$ be a proper metric space. The $K$-homology $K_0(Z)$ and $K_1(Z)$ are groups generated by certain cycles modulo a certain equivalence relation.
\begin{enumerate}[(1)]
    \item A cycle for $K_0(Z)$ is a pair $(H_Z, F)$ where $H_Z$ is a $Z$-module and $F$ is a bounded linear operator on $H_Z$ such that $f(F^*F-1)$, $f(FF^*-1)$ and $fF-Ff$ are compact operators for all $f \in C_0(Z)$;
    \item A cycle for $K_1(Z)$ is a pair $(H_Z, F)$ where $H_Z$ is a $Z$-module and $F$ is a bounded self-adjoint operator on $H_Z$ such that $f(F^2-1)$ and $fF-Ff$ are compact operators for all $f \in C_0(Z)$.
\end{enumerate}

Note that the $Z$-modules can be chosen to be ample in the above description of cycles.

Next, we shall recall the assembly map
 $$\mu: K_i(Z) \to K_i(C^*(Z)).$$
An open cover $\{U_j\}_{j\in J}$ of $Z$ is locally finite if every point $x\in Z$ is contained in only finitely many elements of the cover $\{U_j\}_{j\in J}$.
Let $\{U_j\}_{j\in J}$ be a locally finite and uniformly bounded open cover of $X$. Let $\{\phi_j\}_{j\in J}$ be a partition of unity subordinate to the open cover $\{U_j\}_{j\in J}$. Let $(H_Z,F)$ be a cycle for $K_0(Z)$ such that $H_Z$ is an ample $Z$-module. Define
  $$F'=\sum_j \phi _j ^{\frac{1}{2}}F\phi _j^{\frac{1}{2}},$$
where the infinite sum converges in the strong operator topology. Note that the cycle
$(H_Z,F')$ is equivalent to $(H_Z,F)$ via the homotopy $(H_Z,(1-t)F+tF')$, where $t\in[0,1]$. Note that $F'$ has finite propagation, so that $F'$ is a multiplier of $C^*(Z)$. It is obvious that $F'^2-1 \in C^*(Z)$. Thus, $F'$
is invertible modulo $C^*(Z)$. Hence $F'$ gives rise to an element, denoted by $\partial{[F']}$, in $K_0(C^*(Z))$, where $$\partial:K_1(M(C^*(Z))/C^*(Z))\rightarrow K_0(C^*(Z))$$ is the boundary map in $K$-theory, and $M(C^*(X))$ is the multiplier algebra of $C^*(Z)$. We define {\em the coarse assembly map} $\mu$ from $K_0(Z)$ to $K_0(C^*(Z))$ by
 $$\mu([H_Z,F])=\partial{[F']}.$$
Similarly, we can define the coarse assembly map from $K_1(Z)$ to $K_1(C^*(Z))$.

\subsection{The coarse Baum--Connes conjecture for a discrete metric space with bounded geometry}

Recall that a proper, discrete metric space $X$ is said to have bounded geometry if for any $r>0$ there exists $N>0$ such that any ball of radius $r$ in $X$ contains
at most $N$ elements.
\begin{definition}
Let $X$ be a metric space with bounded geometry. For each $d\geq 0$, the Rips complex $P_d(X)$ at scale $d$ is defined to be the simplicial complex where the set of vertices is $X$, and a finite subset $\{x_0, x_1, \cdots, x_n\} \subseteq X$ spans a simplex if and only if $d(x_i, x_j) \leq d$ for all $0\leq i, j \leq n.$
\end{definition}
The Rips complex $P_d(X)$ is equipped with the spherical metric that is the maximal metric whose restriction to each simplex $\left\{\sum_{i=0}^n t_i x_i: t_i\geq 0, \sum_it_i=1\right\}$ is the metric obtained by identifying the simplex with $S_+^n$ via the map:
$$\sum_{i=0}^n t_i x_i \rightarrow \left(\frac{t_0}{\sqrt{\sum_{i=0}^n t_i^2}}, \cdots, \frac{t_n}{\sqrt{\sum_{i=0}^n t_i^2}}\right),$$
where $S_+^n=\left\{(x_1,\cdots,x_n) \in \mathbb{R}^{n+1}: x_i \geq 0, \sum_{i=0}^n x_i^2=1 \right\}$ endowed with the standard Riemannian metric. The distance between a pair of points in different connected components is defined to be infinity.
	
For each $r\leq s$, note that $P_{r}(X)\subseteq P_{s}(X)$. Then the canonical inclusion map $P_{r}(X) \xhookrightarrow{i_{sr}} P_{s}(X)$ induces a $*$-homomorphism from $C^*(P_{r}(X))$ to $C^*(P_{s}(X))$. Furthermore, for $d\leq r \leq s$, the inclusion map $i_{sd}$ and the composition $i_{sr}\circ i_{rd}$ induce the same homomorphism from $K_*(C^*(P_{d}(X)))$ to $K_*(C^*(P_{s}(X)))$.

\noindent \textbf{The coarse Baum--Connes conjecture.} For any discrete metric space $X$ with bounded geometry, the coarse assembly map
 $$\mu: \lim_{d\to \infty}K_*(P_d(X)) \to  \lim\limits_{d\to \infty}K_*(C^*(P_d(X))) \cong K_*(C^*(X))$$
 is an isomorphism.

Now, we shall recall the definition of localization algebras \cite{Yu-97} and the relation between its $K$-theory and the $K$-homology groups.

\begin{definition}
Let $Z$ be a proper metric space.
\begin{enumerate}[(1)]
 \item The algebraic localization algebra, denoted by $\mathbb{C}_L[Z]$, is defined to be the $*$-algebra of all bounded and uniformly norm-continuous functions $g: [1,\infty)\rightarrow \mathbb{C}[Z]$ such that ${\rm Prop}_{Z}(g(t))\rightarrow 0 \text{ as } t\rightarrow \infty$.
 \item The localization algebra $C^*_L(Z)$ is the norm closure of $\mathbb{C}_L[Z]$ under the norm $$\|f\|=\sup\limits_{t\in[0,\infty)}\|f(t)\|.$$
\end{enumerate}
\end{definition}

Naturally, the evaluation-at-one map
$$e:C_{L}^*(Z) \to C^*(Z)$$
defined by
$$e(g)=g(1)$$
for $g\in C_{L}^*(Z)$ is a $*$-homomorphism.

Next, we recall some basic properties of localization algebras. A Borel map $f:Z_1\rightarrow Z_2$ is said to be Lipschitz, if there exists a positive constant $C$ such that $d(f(x),f(y))\leq Cd(x,y)$ for all $x,y\in Z_1$. Following the construction in \cite{Yu-97}, a Lipschitz map $f:Z_1\rightarrow Z_2$ induces a $*$-homomorphism $\mbox{Ad}(V_f):C_{L}^*(Z_1) \rightarrow C_{L}^*(Z_2)$ as follows.

Let $\{\epsilon_n\}_{n\in\mathbb{N}}$ be a sequence of positive numbers with $\lim\limits_{n\rightarrow \infty}\epsilon_n=0$. For each $k$, there exists an isometry $V_k:H_{Z_1} \rightarrow H_{Z_2}$ such that $\mbox{Supp}(V_k)\subset \{(x,y)\in Z_1 \times Z_2: d(y,f(x))\leq \epsilon_k\}$.

Define a family of isometries
$$(V_f(t))_{t\in[0,\infty)}:H_{Z_1} \oplus H_{Z_1} \rightarrow H_{Z_2}\oplus H_{Z_2}$$
by
$$V_f(t)=R(t-k+1)(V_k\oplus V_{k+1})R^*(t-k+1),$$
for $t\in [k-1,k)$, where $$R(t)=\left(\begin{array}{cc}
cos(\pi t/2) & sin(\pi t/2)\\
-sin(\pi t/2) & cos(\pi t/2)
\end{array}\right).$$
Then $V_f(t)$ induces a homomorphism on unitizations $$\mbox{Ad}(V_f(t)):C_L^*(Z_1)^+\rightarrow C_L^*(Z_2)^+\otimes M_2(\mathbb{C})$$
by $$\mbox{Ad}(V_f(t))(u(t)+cI)=V_f(t)(u(t)\oplus 0)V_f^*(t)+cI$$
for all $u(t)\in C_L^*(Z_1^+)$ and $c\in \mathbb{C}$.

\begin{definition}
	Let $Z_1$ and $Z_2$ be two proper metric spaces and $f,g$ two Lipschitz maps from $Z_1$ to $Z_2$.  A continuous homotopy $F(t,x):[0,1]\times Z_1 \rightarrow Z_2$ between $f$ and $g$ is said to be strongly Lipschitz if
	\begin{itemize}
		\item[(1)] $F(t,x)$ is coarse map from $Z_1$ to $Z_2$ for each $t$;
		\item[(2)] there exists a positive constant $C$ such that $d(F(t,x),F(t,y))\leq Cd(x,y)$ for all $x,y\in Z_1$ and $t\in [0,1]$;
	    \item[(3)] for any $\epsilon >0$, there exists $\delta >0$ such that $d(F(t_1,x),F(t_2,x))<\epsilon$ for all $x\in Z_1$ and $|t
		_1-t_2|<\delta$;
		\item[(4)] $F(0,x)=f(x) \text{ and } F(1,x)=g(x)$ for all $x\in Z_1$.
	\end{itemize}
\end{definition}

\begin{definition}A metric space $Z_1$ is said to be strongly Lipschitz homotopy equivalent to $Z_2$, if there exist two Lipschitz maps $f:Z_1\rightarrow Z_2$ and $g:Z_2\rightarrow Z_1$ such that $f\circ g$ and $g\circ f$ are strongly Lipschitz homotopy equivalent to the identity maps $id_{Z_1}$ and $id_{Z_2}$, respectively.
\end{definition}

The $K$-theory of localization algebras is invariant under strongly Lipschitz homotopy equivalence (see \cite{Yu-97}). The following Mayer--Vietoris sequence was introduced by the third author, and more details can be found in \cite{Yu-97}.

\begin{proposition}
	Let $Z$ be a simplical complex endowed with the spherical metric, and $Z_1, Z_2 \subset Z$ its simplical subcomplexes endowed with the subspace metric. Then we have the following six-term exact sequence:
	$$\xymatrix{
	K_0(C_{L}^*(Z_1\cap Z_2)) \ar[r] & K_0(C_{L}^*(Z_1))\oplus K_0(C_{L}^*(Z_2)) \ar[r] & K_0(C_{L}^*(Z_1\cup Z_2)) \ar[d]\\
	K_1(C_{L}^*(Z_1\cup Z_2)) \ar[u] & K_1(C_{L}^*(X_1))\oplus K_1(C_{L}^*(Z_2))
\ar[l] & K_1(C_{L}^*(Z_1\cap Z_2)) \ar[l]
}$$
\end{proposition}

We then recall a local index map from the $K$-homology group $K_*(Z)$ to the $K$-theory group $K_*(C_L^*(Z))$. For every positive integer $n$, let $\{U_{n,i}\}_i$ be a locally finite open cover for $Z$ with $\mbox{diameter}(U_{n,i})\leq \frac{1}{n}$ for all $i$. Let $\{\phi_{n,i}\}_i$ be the partition of unity subordinate to the open cover $\{U_{n,i}\}_i$. For any $[H_Z, F] \in K_0(Z)$, we define a family of operators $(F(t))_{t\in [0,\infty)}$ by
$$F(t)=\sum_i\left((n-t)\phi_{n,i}^{\frac{1}{2}}F\phi_{n,i}^{\frac{1}{2}}+(t-n+1)\phi_{n+1,i}^{\frac{1}{2}}F\phi_{n+1,i}^{\frac{1}{2}})\right),$$
for all $t\in[n-1,n]$, where the infinite sum converges under the strong operator topology. Since $\mbox{diameter}(U_{n,i}) \rightarrow 0$ as $n\rightarrow \infty$, we have that propagation$(F(t))\rightarrow 0$ as $t\rightarrow 0$. Moreover, we have that the path $(F(t))_{t \in [0,\infty)}$ is a multiplier of $C_L^*(Z)$ and it is a unitary modulo $C_L^*(Z)$. Hence we can define a local index map
$${\rm Ind}_L:K_0(Z)\rightarrow K_0(C_L^*(Z))$$
by
$${\rm Ind}_L([H_Z,F])=\partial([F(t)]),$$
where $\partial:K_1(M(C_{L}^*(Z))/C_L^*(Z))\rightarrow K_0(C_{L}^*(Z))$ is the boundary map in $K$-theory, and $M(C_{L}^*(Z))$ is the multiplier algebra of $C_{L}^*(Z)$.
Similarly we can define the local index map
$${\rm Ind}_L: K_1(Z)\rightarrow K_1(C_{L}^*(Z)).$$
The following result establishes the relation between the $K$-homology groups and the $K$-theory of localization algebras.
\begin{theorem}{\cite{Yu-97}}
For any finite-dimensional simplical complex $Z$ endowed with the spherical metric, the local index map
$${\rm Ind}_L:K_*(Z)\rightarrow K_*(C_{L}^*(Z))$$ is an isomorphism.
\end{theorem}


If $X$ is a discrete metric space with bounded geometry, we have the following commutative diagram:
\begin{equation*}
\xymatrix{
	\ \ &\lim\limits_{d\to \infty}K_*(C_{L}^*(P_d(X))) \ar[d]^{e_*}\\	
	\lim\limits_{d \to \infty}K_*(P_d(X)) \ar[ur]^{{\rm Ind}_L}_{\cong} \ar[r]^{\mu} & \lim\limits_{d \to \infty}K_*(C^*(P_d(X))).
}
\end{equation*}
Therefore, the coarse Baum--Connes conjecture is a consequence of the result that the map
$$e_*: \lim\limits_{d\rightarrow \infty}K_*(C_{L}^*(P_d(X))) \rightarrow \lim\limits_{d\rightarrow \infty}K_*(C_{}^*(P_d(X))) $$
induced by the evaluation-at-one map on $K$-theory is an isomorphism.

\subsection{The coarse Baum--Connes conjecture for a sequence of metric spaces}
In this subsection, we shall first formulate the coarse Baum--Connes conjecture for a sequence of discrete metric spaces with bounded geometry, and then discuss the connection between the coarse Baum-Connes conjecture of a sequence of finite metric spaces and that of their coarse disjoint union.

Let $\{(X_m, d_m)\}_{m \in \mathbb{N}}$ be a sequence of discrete metric spaces with uniform bounded geometry in the sense that for each $R>0$, the number of the elements in the set $B_{X_m}(x,R)=\left\{y \in X_m: d(x,y)\leq R\right\}$ is at most $M_R$ for some $M_R>0$ which does not depend on $m$. Note that a sequence of finitely generated groups $\left(G_m\right)_m$ has uniform bounded geometry if the numbers of generating subsets for all groups are uniformly bounded.

For each $d>0$, and each $m \in \mathbb{N}$, we choose a countable dense subset $X^m_d \subseteq P_d(X_m)$ such that $X^m_d \subseteq X^{m}_{d'}$ for each $m$ if $d<d'$. Fix an infinite-dimensional separable Hilbert space $H$, and denote by $K$ the algebra of compact operators on $H$.
\begin{definition}\label{def-alg-uniform-Roe}
For each $d>0$, the algebraic Roe algebra $\mathbb{C}_u[(P_d(X_m))_{m \in \mathbb{N}}]$ is the collection of tuples $T=\left(T^{m}\right)_{m\in \mathbb{N}}$
satisfying
\begin{enumerate}[(1)]
    \item each $T^{m}: X^m_d \times X^m_d \to K$ by $(x,y)\mapsto T^m_{x,y}$ is a bounded function;
    \item there exists $r>0$ such that for each $m$, $T^{m}_{x,y}=0$ for all $x, y \in X^m_d$ with $d(x,y)\geq r$;
    \item for each $m$ and each bounded subset $B\subseteq P_d(G_m)$, the set
    $$\left\{(x, y) \in (B \times B) \cap (X^m_d\times X^m_d):  T^{m}_{x,y}\neq 0\right\}$$
    is finite.
\end{enumerate}
\end{definition}

We  view $\mathbb{C}_u[(P_d(X_m))_{m \in \mathbb{N}}]$ as a $*$-subalgebra of $\prod_m C_u^*(P_d(X_m))$. Let
$$E=\bigoplus_m \ell^2(X^m_d) \otimes H.$$
The algebraic Roe algebra $\mathbb{C}_u[(P_d(X_m))_{m \in \mathbb{N}}]$ admits a $*$-representation on $E$ by the restriction of the $*$-representation of $\prod_m C^*(P_d(X_m))$.

The Roe algebra for the sequence $\left(P_d(G_m)\right)_{m \in \mathbb{N}}$, denoted by $C_u^*((P_d(X_m))_{m \in \mathbb{N}})$, is the completion of $\mathbb{C}_u[(P_d(X_m))_{m \in \mathbb{N}}]$ under the operator norm on $E$.

We can also define the localization algebra for a sequence of metric spaces.
\begin{definition}
The algebraic localization algebra, $\mathbb{C}_{u,L}[(P_d(X_m))_{m \in \mathbb{N}}]$, is the $*$-algebra of all bounded and uniformly continuous functions
$$f: [1,\infty) \to \mathbb{C}_u[(P_d(X_m))_{m \in \mathbb{N}}]$$
such that $f(t)$ is of the form $f(t)=(f^{m}(t))_{m \in \mathbb{N}}$ for all $t \in [1,\infty)$, where the path of the tuples $(f^{m}(t))_{m \in \mathbb{N}}$ satisfies that
$$
\sup_{m \in \mathbb{N}} {\rm Prop}_{P_d(G_m)}(f^m(t))\to 0, \mbox{~~as~~} \to \infty.
$$

The localization algebra
$C_{u,L}^*((P_d(X_m))_{m \in \mathbb{N}})$ is defined to be the completion of $\mathbb{C}_{u,L}[(P_d(X_m))_{m \in \mathbb{N}}]$
under the norm
$$\|f\|=\sup_{t\in [1, \infty)}\|f(t)\|_{C_u^*((P_d(X_m))_{m \in \mathbb{N}})}$$
for all $f \in \mathbb{C}_{u,L}[(P_d(X_m))_{m \in \mathbb{N}}]$.
\end{definition}
Naturally, we have a $*$-homomorphism
$$e:C_{u,L}^*((P_d(X_m))_{m \in \mathbb{N}})\to C_u^*((P_d(X_m))_{m \in \mathbb{N}})$$
defined by $e((f_m(t))_{m \in \mathbb{N}})=(f_m(1))_{m \in \mathbb{N}}$ for all $(f_m(t))_{m \in \mathbb{N}} \in C_{u,L}^*((P_d(X_m))_{m \in \mathbb{N}})$. Moreover, this homomorphism induces a map
$$e_*: K_*(C^*_{u,L}((P_d(X_m))_{m \in \mathbb{N}})) \to K_*(C_u^*((P_d(X_m))_{m \in \mathbb{N}}))$$
at the $K$-theory level.

\noindent \textbf{The coarse Baum--Connes conjecture for a sequence of metric spaces.} For any  sequence of discrete metric spaces $\left(X_m,d_m\right)_{m \in \mathbb{N}}$ with uniform bounded geometry, the map
 $$e_*: \lim\limits_{d \to \infty} K_*(C^*_{u,L}((P_d(X_m))_{m \in \mathbb{N}})) \to \lim\limits_{d \to \infty} K_*(C_u^*((P_d(X_m))_{m \in \mathbb{N}}))$$
induced by the evaluation-at-one map on $K$-theory is an isomorphism.

Let $\left(X_m,d_m\right)_{m \in \mathbb{N}}$ be a sequence of metric spaces with uniform bounded geometry. The sequence $(X_m,d_m)_{m \in \mathbb{N}}$ is said to be coarsely embeddable into Hilbert space if there exist two non-decreasing functions $\rho_-,\rho_+:[0,\infty) \to [0,\infty)$, and for each $m \in \mathbb{N}$ there exists a map $f_m$ from $X_m$ to a Hilbert space $H_m$ such that
\begin{itemize}
    \item $\lim\limits_{t\to \infty} \rho_{\pm}(t)=\infty$
    \item $\rho_-(d(x, y)) \leq \|f_m(x)-f_m(y)\|_{H_m} \leq \rho_+(d(x,y))$
for all $x, y \in X_m$.
\end{itemize}
Note that a sequence $\left( X_m\right)_{m \in \mathbb{N}}$ is coarsely embeddable into Hilbert space if and only if the metric space of separated disjoint union of the sequence $\left(X_m\right)_{m \in \mathbb{N}}$ is coarsely embeddable into Hilbert space. By the third author's result in \cite{Yu2000}, the coarse Baum--Connes conjecture for the sequence $\left(X_m\right)_{m \in \mathbb{N}}$ is true when the sequence $\left(X_m\right)_{m \in \mathbb{N}}$ can be coarsely embedded into Hilbert space.

Let $(1\to N_m\to G_m\to Q_m \to 1)_{m \in \mathbb{N}}$ be a sequence of short exact sequence of finite groups.
For each $m$, let $S_m$, $S'_m$ and $S''_m$ be the finite symmetric generating subsets of $N_m$, $G_m$ and $Q_m$, respectively, such that
\begin{enumerate}
    \item[(1)] $S_m\subseteq S'_m$ and $\pi_m(S'_m)=S''_m$, where $\pi_m:G_m \to Q_m$ is the quotient map.
    \item[(2)] $\sup_m|S_m|<\infty$, and $\sup_m|S'_m|<\infty.$
\end{enumerate}
The coarse disjoint union $(X, d_X)$ of $(G_m)_{m \in \mathbb{N}}$ is the disjoint union $\bigsqcup_m G_m$ as a set endowed with  a metric $d_X$ satisfying the following conditions:
\begin{enumerate}[(1)]
    \item for each $m$, $d_X(x, y)=d_{m}(x,y)$ for all $x, y\in G_m$;
    \item $\mbox{dist}(G_m,G_{m'})\to \infty$, as $m+m'\to \infty$ for all $m\neq m'$.
\end{enumerate}
It is easy to verify that any two coarse disjoint unions are coarsely equivalent.

The separated disjoint union $(Y, d_Y)$ is the disjoint union equipped with the metric $d_Y$ satisfying the following conditions:
\begin{enumerate}[(1)]
    \item for each $m$, $d_Y(x, y)=d_m(x,y)$ for all $x, y \in G_m$;
    \item $\mbox{dist}(G_m, G_{m'})=\infty$ for all $m \neq m'$.
\end{enumerate}
Note that the coarse disjoint union of a sequence $(G_m)_{m \in \mathbb{N}}$ is coarsely embeddable into Hilbert space if and only the separated union of $(G_m)_{m \in \mathbb{N}}$ is coarsely embeddable.

Now, we discuss the connection between the coarse Baum--Connes conjecture for the coarse disjoint union of a sequence of finite metric spaces and that of their separated disjoint union. The proof of the following result can be found in \cite{Willett-Yu-Book}.
\begin{theorem}\label{theorem-CBC-for-co-disjoint-union-and-s-union}
Let $\left( G_m\right)_{m \in \mathbb{N}}$ be a sequence of finite groups endowed with word length metrics such that the metric spaces $X$ and $Y$ defined above have bounded geometry.
If the coarse Baum--Connes conjecture holds for the separated disjoint union $(Y, d_Y)$, then the coarse Baum--Connes conjecture holds for the coarse disjoint union $(X, d_X)$.
\end{theorem}

To prove the coarse Baum--Connes conjecture for a coarse disjoint union, it suffices to prove the coarse Baum--Connes conjecture for the family of metric spaces.
Theorem \ref{thm:main result} follows from the following result.
\begin{theorem}\label{theorem-CBC-for-sequence}
Let $(1 \to N_m \to G_m \to Q_m \to 1)_{m \in \mathbb{N}}$ be a sequence of exact sequences of finite groups such that the sequences $(N_m)_m$, $(G_m)_m$ and $(Q_m)_m$ have uniform bounded geometry. If $(N_m )_m$ and $(Q_m )_m$ are coarsely embeddable into Hilbert space, then the coarse Baum--Connes conjecture holds for the sequence $(G_m)_{m \in \mathbb{N}}$, that is, the map
 $$e_*: \lim\limits_{d \to \infty} K_*(C^*_{u,L}((P_d(G_m))_{m \in \mathbb{N}})) \to \lim\limits_{d \to \infty} K_*(C_u^*((P_d(G_m))_{m \in \mathbb{N}}))$$
 induced by the evaluation-at-one map on $K$-theory is an isomorphism.
\end{theorem}
In the rest of this paper, we will focus on proving the above result.

\section{Twisted Roe algebras and localization algebras}

In this section, we shall first recall some properties of Bott--Dirac operator on finite-dimensional Euclidean spaces, and then use the coarse embeddability of the the sequence $(Q_m)_{m \in \mathbb{N}}$ to formulate the twisted coarse Baum--Connes conjecture for the sequence $(G_m)_{m \in \mathbb{N}}$.

\subsection{Preliminaries}\label{sect:prelimilary on differential geometry}
Let $V$ be an even-dimensional Euclidean space. Following the constructions of Willett--Yu \cite{Willett-Yu-Book}, we shall recall the concepts of the Hilbert space $\mathcal{L}^2_V$ and some unbounded, essentially self-adjoint operator $F$ on $\mathcal{L}^2_V$.
This operator plays a crucial role in the definition of twisted Roe algebras which encode the geometry of the metric spaces and reduce the coarse Baum--Connes conjecture to the twisted coarse Baum--Connes conjecture. Using the techniques developed in \cite{Willett-Yu-Book}, we can bypass the difficulties caused by the obstruction that the maximal twisted Roe algebras may not be  isomorphic to the reduced twisted Roe algebras \cite{Yu2000}.

Denote by ${\rm Cliff}_{\mathbb{C}}(V)$ the (complexified) Clifford algebra, which is the universal unital complex algebra containing $V$ as a subspace and subject to the multiplicative relations
$$
x\cdot x=\|x\|^2
$$
for all $x \in V$, where $\|x\|$ is the norm of $x$ in the Euclidean space. In this paper, ${\rm Cliff}_{\mathbb{C}}(V)$ is viewed as a graded $C^*$-algebra as well as a graded Hilbert space, denoted by $H_V$, where both gradings are given by the grading on the Clifford algebra.

Let $\mathcal{L}^2_V$ be the Hilbert space of square integrable functions from $V$ to $H_V$ and let $\mathcal{S}_V$ be the subspace of all Schwartz functions from $V$ to $H_V$.

Now, fix an orthonormal basis $\{e_1,e_2,\cdots, e_d\}$ of $E$, and let $x_1,x_2,\cdots, x_d: V \to \mathbb{R}$ be the corresponding coordinates. The Bott operator
$$
C:\mathcal{L}^2_V\to \mathcal{L}^2_V
$$
is defined by
$$
(C\xi)(v)=\sum_{i=1}^d (x_ie_i)\cdot \xi(v)
$$
for all $\xi \in \mathcal{L}^2_V$ and $v=\sum_{i=1}^d x_ie_i \in V$. The Dirac operator
$$
D:\mathcal{L}^2_V\to \mathcal{L}^2_V
$$
is defined by
$$
D\xi(u)=\sum_{i=1}^d \widehat{e}_i \frac{\partial \xi}{\partial x_i}(u)
$$
for all $\xi \in \mathcal{L}^2_V$. It is a well-known fact that the Bott operator and the Dirac operator are unbounded and self-adjoint operators with domain $\mathcal{S}_V$. The Bott--Dirac operator is the unbounded operator
$$B=D+C:\mathcal{L}^2_V \to \mathcal{L}^2_V$$
with domain $\mathcal{S}_V$. Note that $B$ is essentially self-adjoint. For $x \in V$, let $c_x$ be the operator on $\mathcal{L}^2_V$ defined by the Clifford multiplication by the vector $x$. Note that $c_x$ is a bounded, self-adjoint operator on $\mathcal{L}^2_V$. To study the twisted coarse Baum--Connes conjecture, we need the following operator $F$ on $\mathcal{L}^2_V$.

\begin{definition}\label{defn:Operator F}
Let $s \in [1,\infty)$ and $x\in V$.
\begin{enumerate}
\item[(1)] The Bott--Dirac operator associated with $(s, x)$ is the unbounded operator
$$
B_{s,x}=s^{-1}D+C-c_x.
$$
\item[(2)] Define the bounded operator on $\mathcal{L}^2_V$
$$F_{s,x}=B_{s,x}(1+B^2_{s, x})^{-1/2}.$$
\end{enumerate}

\end{definition}

For $R>0$ and $x \in V$, let $\chi_{R, x}$ be the indicator function on the closed $R$-ball of radius $R$ centered at $x$.
Note that the Hilbert space $\mathcal{L}^2_V$ is an ample $V$-module. Each function $f \in C_0(V)$ can be viewed as a function from $V$ to the Clifford algebra ${\rm Cliff}_{\mathbb{C}}(V)$. The $*$-representation of $C_0(V)$ on $\mathcal{L}^2_{V}$ is defined by
$$f(g)(v)=f(v)\cdot g(v)$$
for all $f\in C_0(V)$, $g \in \mathcal{L}^2_{V}$ and $v\in V$. The following properties will be useful later, and the proof can be found in \cite{Willett-Yu-Book}.
\begin{proposition}\label{prop:property of F}
For each $\epsilon>0$, there exists an odd function $\Psi: \mathbb{R}\to [-1,1]$ satisfying $\Psi(t)\to \pm 1$ as $t \to \pm \infty$ with the following properties:
\begin{enumerate}[(1)]
    \item for all $s$ and $x$, $\|F_{s,x}-\Psi(B_{s,x})\|<\epsilon$;
    \item there exits $R_0>0$ such that
          $${\rm Prop}_V(\Psi(B_{s,x}))\leq s^{-1}R_0$$
          for all $s \in [1, \infty)$ and $x \in V$, where the $V$-propagation of $T$ is defined by viewing the Hilbert space $\mathcal{L}^2_V$ as a $V$-module by the standard way;
    \item for all $s \in [1, \infty)$ and $x \in V$, the operator $\Psi(B_{s,x})^2-1$ is compact;
    \item for all $s \in [1, \infty)$ and $x, y \in V$, the operator $\Psi(B_{s, x})-\Psi(B_{s,y})$ is compact;
    \item there exists a constant such that for all $s \in [0, \infty)$ and $x, y \in V$,
          $$
          \|\Psi(B_{s,x})-\Psi(B_{s,y})\|\leq c|x-y|;
          $$
    \item for all $x \in V$, the function
          $$[1, \infty)\to B(\mathcal{L}^2_V), \quad s \mapsto \Psi(B_{s, x})$$
          is strongly-$*$ continuous;
    \item For any $b>0$, the family of functions
    $$
    [1,b]\to B(\mathcal{L}^2_{V}),~~~s\mapsto \Psi(B_{s,x})^2-1
    $$
    is norm equicontinuous in $x$ and $s$;
    \item for any $r>0$, denote by $E_r=\{(x,y)\in V\times V:  |x-y|\leq r\}$. Then for any $b>0$, the family of functions
    $$
    [1,b]\to B(\mathcal{L}^2_{V}), s \mapsto \Psi(B_{s,x})-\Psi(B_{s,y})
    $$
    is norm equicontinuous on $(x,y)\in E_r$ and $s\in [1,b]$;
    \item there exists $R_1>0$ such that for all $R>R_1$, $s\in [1,\infty)$ and $x\in V$, we have that
     $$\|(\Psi(B_{s,x})^2-1)(1-\chi_{x, R})\|<\epsilon;$$
    \item for all $r>0$, there exists $R_1>0$ such that for all $R\geq R_1$ and all $s \in [2d, \infty)$ and $x, y \in V$ with $|x-y|\leq r$, we have that
    $$\|(\Psi(B_{s,x})-\Psi(B_{s,y}))(1-\chi_{x, R})\|<\epsilon.$$
\end{enumerate}
\end{proposition}
The operators $F_{s,x}$ will be used to define the index maps later which help to reduce the Baum--Connes conjecture to a twisted Baum--Connes conjecture.

\subsection{Twisted Roe algebras and twisted localization algebras}

In this subsection, we shall first define the twisted Roe algebra using the geometry of the sequence $(Q_m)_{m\in \mathbb{N}}$. Then we prove the twisted version of coarse Baum--Connes conjecture for the sequence $(G_m)_{m\in \mathbb{N}}$.

Let $(\varphi_m: Q_m \to V_m )_{m\in \mathbb{N}}$ be a sequence of coarse embeddings into Euclidean spaces, where each $V_m$ is an Euclidean space of even dimension $d_m$. We first lift each coarse embedding $\varphi_m: Q_m \to V_m$ to a map on $G_m$ along the quotient map $\pi_m: G_m\to Q_m$, and then extend this map to the Rips complex $P_d(G_m)$, still denoted by $\varphi_m: P_d(G_m)\to V_m$, mapping $\sum_{i}t_ix_i$ to $\sum_{i}t_i\varphi_m(\pi_{m}(x_i))$ for each $d>0$.

Recall that we fixed a countable dense subset $X^m_d \subseteq P_d(G_m)$ for each $m$ and each $d$ in the definition of Roe algebra for the sequence $(G_m)_{m\in \mathbb{N}}$. Fix an infinite-dimensional separable Hilbert space $H$. Consider the tensor product Hilbert space
$$H_{d,V_m}=\ell^2(X^m_d)\otimes H \otimes \mathcal{L}^2_{V_m}.$$
Denote by $K(H_{d,V_m})$ and $B(H_{d,V_m})$ the algebra of compact operators and the algebra of all bounded linear operators, respectively.
Note that $H_{d,V_m}$ is both an ample $P_d(G_m)$-module and an ample $V_m$-module for each $m$. Therefore, for each bounded operator $T$ on the geometric module $H_{d,V_m}$, we can consider the $P_d(G_m)$-propagation and the $V_m$-propagation for $T$.

For every bounded linear operator $T$ on $\mathcal{L}^2_{V_m}$, we can define an operator $T^V$ on the Hilbert space $H_{d,V_m}$ as follows.
\begin{definition}
For each $m\in \mathbb{N}$ and each operator $T\in B(\mathcal{L}^2_{V_m})$, we define a bounded operator $T^{V_m}$ on $H_{d,V_m}$ by
$$
T^{V_m}:\delta_x\otimes \xi\otimes u\mapsto \delta_x\otimes \xi\otimes V_{\varphi_m(x)}TV^*_{\varphi_m(x)}u,
$$
for all $x \in X^m_d$, $\xi\in H$ and $u \in \mathcal{L}^2_{V_m}$.
\end{definition}

For each tuple $(T_m)_{m \in \mathbb{N}}\in \prod_m B(\mathcal{L}^2_{V_m})$, by the definition, we know that the operator $T^{V}$ is a block-diagonal operator and the entries of the operator is given by:
\begin{equation*}
T^{V}_{m,x,y}= \left\{
        \begin{array}{ll}
            {\rm Id}\otimes V_{\varphi_m(x)}T_mV^*_{\varphi_m(x)}, & \quad y=x \\
            0 & \mbox{otherwise}.
        \end{array}
    \right.
\end{equation*}
For each $m \in \mathbb{N}$, let $\chi_{V_m,R}$ be the characteristic function on the closed ball of radius $R$ in $V_m$ centered at the origin. The operators $\chi^V_{V_m,R}$ are used to encode the coarse embeddings $(\varphi_m)_{m \in \mathbb{N}}$ when defining the twisted Roe algebras.

On the ample $P_d(G_m)$-module, we can define the Roe algebra $C^*(P_d(G_m);H_{d,V_m})$ for each $m \in \mathbb{N}$ and $d>0$. Since each $P_d(G_m)$ is a bounded metric space, by the local compactness, each element $T$ in $\mathbb{C}[P_d(G_m);H_{d,V_m}]$ can be approximated by an $X^m_d$$\times$$X^m_d$-matrix $(T_{x,y})_{x,y\in X^m_d}$, where each entry $T_{x,y}$ is a compact operator on $H\otimes \mathcal{L}^2_{V_m}$ and the set $\{(x,y)\vert T_{x,y}\neq 0\}$ is finite.

For each $d>0$, let $\prod_{m} C_b([0,\infty), C^*(P_d(G_m);H_{d,V_m}))$ be the product $C^*$-algebra of all bounded norm-continuous from $[1,\infty)$ to the Roe algebra $C^*(P_d(G_m);H_{d,V_{m}})$ equipped with the supreme norm. Each element of this $C^*$-algebra can be expressed as a tuple $(T^m_{s})_{m \in \mathbb{N}, s \in [1,\infty)}$
where for each $m \in \mathbb{N}$, $s \mapsto T^m_{s}$ is a function from $[1,\infty)$ to $C^*(P_d(G_m);H_{d,V_{m}})$.

Now, let us introduce the twisted Roe algebra following \cite[Section 12.6]{Willett-Yu-Book}.
\begin{definition}\label{Def:TwistedRoeAlgebra}

Let $d>0$. The algebraic twisted Roe algebra $\mathbb{A}((P_d(G_m); V_m)_{m \in \mathbb{N}})$ is defined to be the $*$-subalgebra of the product $\prod_{m} C_b([0,\infty), C^*(P_d(G_m);H_{d,V_m})))$ consisting of elements $(T^m_{s})_{m \in \mathbb{N}, s \in [1,\infty)}$ satisfying the following properties:

\begin{enumerate}[(1)]
    \item $\sup\limits_{m \in\mathbb{N}, s \in [1,\infty)} {\rm Prop}_{P_d(G_m)}(T^m_{s})<\infty$;
    \item $\lim\limits_{s \to \infty}\sup\limits_{m \in \mathbb{N}} {\rm Prop}_{V_m}(T^m_{s})=0$;
    \item $\lim\limits_{R \to \infty}\sup\limits_{m \in \mathbb{N}, s \in[1,\infty)} \|\chi^{V}_{V_m,R} T^m_{s}-T^m_{s}\|=\lim\limits_{R \to \infty}\sup\limits_{m \in \mathbb{N}, s \in[1,\infty)} \|T^m_{s}\chi^{V}_{V_m,R}-T^m_{s}\|=0$.
\end{enumerate}
The twisted Roe algebra $A((P_d(G_m);V_m)_{m \in \mathbb{N}})$ of the sequence $(P_d(G_m))_{m \in \mathbb{N}}$ is defined to be the closure of $\mathbb{A}((P_d(G_m);V_m)_{m \in \mathbb{N}})$ in $\prod_{m} C_b([0,\infty), C^*(P_d(G_m);H_{d,V_m}))$.
\end{definition}

\begin{remark}\label{rmk:projection convergence}
We remark here that though our definition of the twisted Roe algebras seems to be slightly different from the one in \cite[Definition 12.3.4]{Willett-Yu-Book}, they are in fact the same in our situation. Let $\{p_i\}_{i \in I}$ be the net of finite-rank projections on $\mathcal{L}^2_{V_m}$. For any operator $T=(T_{x,y})_{x,y\in X^m_d} \in \mathbb{C}[P_d(G_m);H_{d,V_m}]$, we assume that the set $\{(x,y)\in X^m_d \times X^m_d: T_{x,y}\neq 0\}$ is finite, i.e., $T$ is a finite matrix.
Since $p_i$ converges to identity under the strong-$*$ topology, we have that ${\rm Id}_{H}\otimes V_{\varphi_m(x)}p_i V^*_{\varphi_m(x)}$ converges to ${\rm Id}_{H \otimes \mathcal{L}^2_{V_m}}$ for each $x$. Since $T_{x,y}$ is a compact operator, we have that both $({\rm Id}_{H}\otimes V_{\varphi_m(x)}p_i V^*_{\varphi_m(x)})T_{x,y}$ and $ T_{x,y}({\rm Id}_{H}\otimes V_{\varphi_m(x)}p_i V^*_{\varphi_m(x)})$ converge to $T_{x,y}$ under the norm topology. In addition, $T$ is a finite matrix. It follows that
$$
\lim\limits_{i \to \infty}\|T-Tp_i^V\|=\lim\limits_{i \to \infty}\|T-p_i^VT\|=0.
$$
This recovers Condition (iv) of Definition $12.6.2$ in \cite{Willett-Yu-Book}.
\end{remark}

Essentially, the twisted Roe algebras can be viewed as a localized algebra in the $(V_m)_{m \in \mathbb{N}}$-direction by the condition $(2)$ in Definition \ref{Def:TwistedRoeAlgebra}. Thus, their $K$-theory have similar properties as the localization algebras.

Let $A=(A_m)_m$ be a sequence of closed subsets where each $A_m$ is a closed subset of $V_m$. We can define
$$
A((P_d(G_m);A_m)_{m \in \mathbb{N}})=(1\otimes \chi_A)A((P_d(G_m); V_m)_{m \in \mathbb{N}})(1\otimes \chi_A),
$$
where $1\otimes \chi_A=({\rm Id}_{\ell^2(X^m_d)\otimes H}\otimes \chi_{A_m})_{m \in \mathbb{N}}$ is a sequence of projections. For given two sequences of closed subsets $A=(A_m)_m,B=(B_m)_m$, denote by $A\cap B= (A_m\cap B_m)_{m \in \mathbb{N}}$ and $A\cup B= (A_m\cup B_m)_{m \in \mathbb{N}}$.
Similar to the $K$-theory of the localization algebras, following the proof of \cite[Proposition 3.11]{Yu-97} we have the following six-term exact sequence:
\begin{proposition}\label{prop:SixTermESTwisted}
Let $A=(A_m)_m,B=(B_m)_m$ be two sequences of closed subsets where $A_m, B_m\subseteq V_m$ are closed subsets for all $m \in \mathbb{N}$. Then we have the the following six-term exact sequence:
\begin{equation*}
 \begin{tikzcd}
       K_0(I_{A\cap B})\ar[r] & K_0(I_{A})\oplus K_0(I_{ B})\ar[r]& K_0(I_{A\cup B})\ar[d]\\
       K_1(I_{A\cup B})\ar[u]& K_1(I_{A})\oplus K_1(I_{B})\ar[l]& K_1(I_{A\cap B})\ar[l].
 \end{tikzcd}
\end{equation*}
where we set for brevity that $I_A=A((P_d(G_m); A_m)_{m \in \mathbb{N}})$, $I_B=A((P_d(G_m); B_m)_{m \in \mathbb{N}})$, $I_{A\cap B}=A((P_d(G_m); A_m\cap B_m)_{m \in \mathbb{N}})$ and $I_{A \cup B}=A((P_d(G_m); A_m\cup B_m)_{m \in \mathbb{N}})$
\end{proposition}

Now let us introduce the twisted localization algebras.
\begin{definition}
Let $d>0$. The algebraic twisted localization algebra $\mathbb{A}_L((P_d(G_m); V_m)_{m \in \mathbb{N}})$ is defined to be the $*$-subalgebra of uniformly bounded and uniformly norm-continuous functions
$$T(t)=(T^m_{s}(t))_{m \in \mathbb{N}}:[1,\infty) \to \mathbb{A}((P_d(G_m);H_{d, V_m})_{m \in \mathbb{N}})$$
satisfying that
$$\lim\limits_{t\to \infty}\sup\limits_{m \in\mathbb{N}, s \in [1,\infty)} {\rm Prop}_{P_d(G_m)}(T^m_{s}(t))=0.$$

The twisted localization algebra $A_L((P_d(G_m);V_m)_{m \in\mathbb{N}})$ of the sequence $(P_d(G_m))_{m\in \mathbb{N}}$ is defined to be the closure of $\mathbb{A}_L((P_d(G_m); V_m)_{m \in \mathbb{N}})$ under the norm
$$
\|T\|=\sup_{t \in [1,\infty)} \|T(t)\|_{A((P_d(G_m);V_m)_{m \in\mathbb{N}})}.
$$
\end{definition}

As each element of the twisted Roe algebras is a path in Roe algebras and each element in the twisted localization algebras is a path in localization algebras, we can consider the evaluation maps. For each fixed $s_0 \in [1,\infty)$, we define
$$
\imath^{s_0}:A((P_d(G_m); V_m)_{m\in \mathbb{N}})\to C^*_u((P_d(G_m);H_{d,V_m})_{m\in \mathbb{N}}),~~~(T^m_s)_{m \in \mathbb{N}, s \in [1,\infty)}\mapsto (T^m_{s_0})_{m \in \mathbb{N}},
$$
and
$$
\imath_L^{s_0}:A_L((P_d(G_m); V_m)_{m\in \mathbb{N}})\to C^*_{u,L}((P_d(G_m);H_{d,V_m})_{m\in \mathbb{N}}),~~~(T^m_s(t))_{m \in \mathbb{N}, s,t \in [1,\infty)}\mapsto (T^m_{s_0}(t))_{m \in \mathbb{N},t\in [1,\infty)}.
$$
Denote by $\imath^{s_0}_*$ and $\imath^{s_0}_{L,*}$ the maps induced by $\imath^{s_0}$ and $\imath^{s_0}_{L}$ on $K$-theory, respectively.

Let us now recall some properties of the twisted localization algebras. Let $A=(A_m)_m$ be a sequence of closed subsets where each $A_m$ is a closed subset of $V_m$. We define
$$
A_L((P_d(G_m); A_m)_{m \in \mathbb{N}})=(1\otimes \chi_A)A((P_d(G_m); V_m)_{m \in \mathbb{N}})(1\otimes \chi_A).
$$
We also have the following six-term exact sequence for localization algebras:
\begin{proposition}\label{prop:SixTermESLocalization}
Let $A=(A_m)_m,B=(B_m)_m$ be two sequences of closed subsets where $A_m, B_m\subseteq V_m$ are closed subset for each $m\in  \mathbb{N}$. Then we have the following six-term exact sequence:
\begin{equation*}
   \begin{tikzcd}
       K_0(L_{A\cap B})\ar[r] & K_0(L_{A})\oplus K_0(L_{ B})\ar[r]& K_0(L_{A\cup B})\ar[d]\\
       K_1(L_{A\cup B})\ar[u]& K_1(L_{A})\oplus K_1(L_{B})\ar[l]& K_1(L_{A\cap B})\ar[l],
   \end{tikzcd}
\end{equation*}
 where we set for brevity that $L_A=A_L((P_d(G_m); A_m)_{m \in \mathbb{N}})$, $L_B=A_L((P_d(G_m); B_m)_{m \in \mathbb{N}})$, $L_{A\cap B}=A_L((P_d(G_m); A_m\cap B_m)_{m \in \mathbb{N}})$ and $L_{A \cup B}=A_L((P_d(G_m); A_m\cup B_m)_{m \in \mathbb{N}})$.
\end{proposition}

Now, let us introduce the twisted coarse Baum--Connes conjecture. There is a natural evaluation map
$$
e: A_L((P_d(G_m))_m; V_m)_{m \in \mathbb{N}})\to A((P_d(G_m); V_m)_{m \in \mathbb{N}})
$$
defined by
$$
e((T^m_{s}(t))_{m \in\mathbb{N}, s, t\in [1,\infty)})=(T^m_{s}(1))_{m \in\mathbb{N}, s\in [1,\infty)}
$$
for all $(T^m_{s}(t))_{m \in\mathbb{N}, s, t\in [1,\infty)} \in A_L((P_d(G_m); V_m)_{m \in \mathbb{N}})$. As a result, there is a homomorphism
$$
e_*:\lim\limits_{d \to \infty}K_*( A_L((P_d(G_m); V_m)_{m \in \mathbb{N}})\to \lim\limits_{d \to \infty}K_*(A((P_d(G_m); V_m)_{m \in \mathbb{N}}))
$$
on $K$-theory induced by the evaluation map $e$.
The \emph{twisted Baum--Connes conjecture} claims that the above map $e_*$ is an isomorphism. We shall prove the twisted coarse Baum--Connes conjecture holds for $(G_m)_{m \in \mathbb{N}}$ under the extra condition that the sequence $(N_m)_{m \in \mathbb{N}}$ is coarsely embeddable into Hilbert space.


\subsection{The index maps}\label{Sect: Index map}

In this section, we shall first construct index maps
$$
{\rm Ind}_F: K_*(C^*((P_d(G_m));H_{d,m})_{m \in \mathbb{N}})) \to K_*(A((P_d(G_m); V_m)_{m \in \mathbb{N}})),
$$
and
$$
{\rm Ind^L}_F: K_*(C_{u,L}^*((P_d(G_m);H_{d, m})_{m \in \mathbb{N}})) \to K_*(A_L((P_d(G_m); V_m)_{m \in \mathbb{N}})).
$$

Recall that for each $m \in \mathbb{N}$ and each $d>0$, we define $H_{d,m}=\ell^2(X^m_d)\otimes H$. Let $C^*_u((P_d(G_m); H_{d,
m})_{m \in\mathbb{N}})$ and $C^*_{u,L}((P_d(G_m); H_{d,
m})_{m \in\mathbb{N}})$ be the Roe algebra and localization algebra for the sequence $(P_d(G_m))_{m\in \mathbb{N}}$ with respect to the sequence of ample modules $(H_{d, V_m})_{m \in \mathbb{N}}$.

We have the commutative diagram:
\begin{equation*}
\begin{tikzcd}
    K_*(C^*_{u,L}((P_d(G_m);H_{d,m})_{m \in \mathbb{N}}))\ar[d, "{\rm Ind}^L_F"]\ar[r,"e_*"]& K_*(C^*_{u}((P_d(G_m);H_{d, m})_m))\ar[d,"{\rm Ind}_F"]\\
    K_*(A_L((P_d(G_m); V_m)_{m \in \mathbb{N}}))\ar[r,"e_*"]\ar[d,"\imath^{s}_{L,*}"]& K_*(A((P_d(G_m); V_m)_{m \in \mathbb{N}}))\ar[d, "\imath^s_*"]\\
    K_*(C_{u,L}((P_d(G_m);H_{d,V_m})_{m \in \mathbb{N}}))\ar[r,"e_*"]&K_*(C_{u}((P_d(G_m);H_{d,V_m})_{m \in \mathbb{N}})).
\end{tikzcd}
\end{equation*}
Then we show that the composition maps $\imath_*^s\circ {\rm Ind}_F$ and $\imath_{L,*}^s\circ {\rm Ind}^L_F$ are isomorphism for any $s>0$ where $\imath^s_*$ and $\imath^s_{L,*}$ are maps on $K$-theory induced by the evaluations. As a consequence of diagram chasing, the coarse Baum--Connes conjecture is reduced to the twisted coarse Baum--Connes conjecture. We shall follow the framework introduced by Willett and Yu in \cite[Chapter 12]{Willett-Yu-Book}.

Now, let us construct the index maps.

\begin{definition}\label{def:F^m_s}
For each $s \in [1, \infty)$, let
$F^m_{s, 0}=B^m_{s,0}(1+(B^m_{s,0})^2)^{-1/2}: \mathcal{L}^2_{V_m}\to \mathcal{L}_{V_m}^2$, where $B^m_{s,0}$ is the Bott--Dirac operator on $\mathcal{L}^2_{V_m}$ defined in Section \ref{sect:prelimilary on differential geometry}.
Let $F^m_s$ be the bounded operator on $H_{d,V_m}$
defined by
$$
F^m_{s}=F^m_{s+2d_m,0}=f(B^m_{s+2d_m,0}),
$$
where $f(x)=x/\sqrt{1+x^2}$ is the function defined on $\mathbb{R}$.
Let $F^m$ be the operator on $L^2([1,\infty)\otimes H_{d, V_m})$ defined by
$F^m\xi(s)=F^m_s(\xi(s))$ for all $\xi \in L^2([1,\infty)\otimes H_{d,V_m})$. Denote $F=(F^m)_m$.
For the function $\Psi$ in Proposition \ref{prop:property of F}, we denote by $F^{\Psi}=(F^{\Psi, m})_{m\in \mathbb{N}}$, where $F^{\Psi, m}=\Psi(B^m_{s+2d_m,0})$.
\end{definition}
\begin{remark}\label{rmk:propagations of (F)}
Note that for each $m\in \mathbb{N}$ and $s \in [1,\infty)$, the $V_m$-propagation of $F_s^m$ is zero, while the $V_m$-propagation of $F_s^{\Psi, m}$ is $R_0/(s+2d_m)$ for some fixed $R_0$ Condition $(2)$ in Proposition \ref{prop:property of F}.
\end{remark}
The following result plays a crucial role to detect the continuity of a one-parameter family of operators from associated matrix entries.

\begin{proposition}[\cite{Gong-Wang-Yu Crelle 2008}]\label{prop:Gong-Wang-Yu operator norm estimate}
Let $X$ be a metric space with bounded geometry and $H_X$ an ample $X$-module. For each $R>0$, there exists $N>0$ such that each operator $T=(T_{x,y})_{x, y\in X}\in \mathbb{C}[X, H_X]$ with ${\rm Prop}_X(T)\leq R$ satisfies
$$
\|T\|\leq N\cdot \sup_{x,y \in X} \|T_{x,y}\|.
$$
\end{proposition}

For each $m \in \mathbb{N}$, each $d>0$ and each subset $[a,b]\subseteq [1, \infty)$, let
$$[a,b] \to C^*(P_d(G_m);H_{d,V_m}),~~~~s \mapsto T_s=(T_{s,x,y})_{x,y\in X^m_d}$$
be a uniformly bounded function with $\sup_{s \in [a,b]} {\rm Prop}_{P_d(G_m)}(T_s)<\infty$. As an immediate consequence of Proposition \ref{prop:Gong-Wang-Yu operator norm estimate}, $(T_s)$ is norm continuous if and only if the family of maps $$(s \mapsto T_{s,x,y})_{x,y
\in X^m_d}$$
is norm equicontinuous.

\begin{lemma}\label{lem: F as self-adjoint, odd in multiplier}
The operator $F=(F^m)_{m \in \mathbb{N}}$ is a self-adjoint, norm one, odd operator in
the multiplier algebra of $A((P_d(G_m); V_m)_{m\in \mathbb{N}})$.
\end{lemma}

\begin{proof}
Since each operator $F_s^m$ is self-adjoint, norm one and odd, so is $F=(F^m)_{m\in\mathbb{N}}$. For any $\epsilon>0$, let $\Psi: \mathbb{R} \to [-1,1]$ be the function in Proposition \ref{prop:property of F} such that
$$
\|F^m-F^{\Psi,m}\|\leq\epsilon.
$$
It suffices to show that
$$
TF^{\Psi}=(T^m_{s}F^{\Psi,m}_{s})_m\in A((P_d(G_m);V_m)_{m\in \mathbb{N}}),
$$
for all $(T^m_{s})_{m\in \mathbb{N}}\in A((P_d(G_m);V_m)_{m\in \mathbb{N}})$.

To show that $TF^{\Psi}\in A((P_d(G_m);V_m)_m)$, we need to show first that $s \mapsto T^m_sF^{\Psi}_s$ is norm continuous for each $m \in \mathbb{N}$. Let $b>0$ be any positive number. By Remark \ref{rmk:projection convergence} and the compactness of the subset $[1.b]$, for each $\epsilon>0$, $s\in [1,b]$, and $m\in \mathbb{N}$, there exists a finite rank projection $p$ on $\mathcal{L}^2_{V_m}$ such that
$$
\|T^m_s-T_s^mp^V\|< \epsilon.
$$
Since $\|F_s^{\Psi, m}\|\leq 1$, we have that
$$
\|T^m_sF^{\Psi,m}-T_s^mp^VF^{\Psi,m}\|<\epsilon.
$$
Thus it suffice to show that $s \mapsto T^m_sp^VF^{\Psi,m}_s$ is norm continuous on $[1,b]$ for each $m \in \mathbb{N}$. By Proposition \ref{prop:Gong-Wang-Yu operator norm estimate}, we need to show that each entry function $s \mapsto T^m_{s,x,y}V_{f(y)}p^V\Psi(B_{s+2d_m,0})V^*_{f(y)}$ is uniformly continuous on each compact subset $[1,b]\subseteq [1,\infty)$. This follows from the part $(6)$ of Proposition \ref{prop:property of F}. Therefore, we have that $s \mapsto T^m_sF^m_s$ is a norm continuous function on $[1, \infty)$.

Next, we shall verify the Condition $(1)$ on the $P_d(G_m)$-propagation of $TF$ in Definition \ref{Def:TwistedRoeAlgebra}. For each $m$, the $P_d(G_m)$-propagation of $F^{\Psi,m}_{s}$ is zero. Therefore, we have that
$$
\sup_{s \in [1,\infty)} \sup_m {\rm Prop}_{P_d(G_m)}\left( T^m_sF^{\Psi,m}_s\right)<\infty.
$$
By Remark \ref{rmk:propagations of (F)}, we have that
$$
\lim\limits_{s \to \infty} \sup_m {\rm Prop}_{V_m}\left( F^{\Psi,m}_s\right)=0.
$$
It follows that
$$
\lim\limits_{s \to \infty} \sup_m {\rm Prop}_{V_m}\left( T^m_sF^{\Psi,m}_s\right)=0.
$$
Condition $(3)$ of Definition \ref{Def:TwistedRoeAlgebra} follows from the fact that the $V_m$-propagation of $F^{\Psi,m}_s$ is uniformly bounded at $s$ and $m$. This finishes the proof.
\end{proof}

For each $m\in \mathbb{N}$ and each $d>0$, recall that $H_{d,m}=\ell^2(X^m_d)\otimes H$ and $H_{d,V_m}=\ell^2(X^m_d)\otimes H\otimes \mathcal{L}^2_{V_m}$. The Roe algebra $C^*_m(P_d(G_m))$ can be represented on the Hilbert space $L^2[1, \infty)\otimes H_{d,V_m}$ by
$$C^*(P_d(G_m); H_{d,m})\to B(L^2[1, \infty)\otimes H_{d,V_m}),~~~~T\mapsto {\rm Id}_{L^2[1,\infty)}\otimes T\otimes {\rm Id}_{\mathcal{L}^2_{V_m}}.$$
Accordingly, the Roe algebra $C^*_u((P_d(G_m); H_{d,m})_{m \in \mathbb{N}})$ of the sequence $(P_d(G_m))_{m \in \mathbb{N}}$ with respect to the sequence of modules $(H_{d,m})_{m \in \mathbb{N}}$ can be viewed as a $C^*$-subalgebra of $B\left(L^2([1, \infty) \otimes (\oplus_{m}H_{d,V_m})\right)$. The following results states that $C^*_u((P_d(G_m); H_{d,m})_{m \in \mathbb{N}})$ is actually a subalgebra of the multiplier algebra of $A((P_d(G_m);V_m)_{m \in \mathbb{N}})\subseteq B\left(L^2[1, \infty) \otimes (\oplus_{m} H_{d,V_m})\right)$.

\begin{lemma}\label{lem:twisted Roe subalgebra}
The Roe algebra $C^*_u((P_d(G_m); H_{d,m})_{m \in \mathbb{N}})$ is a $C^*$-subalgebra of the multiplier algebra of $A(P_d(G_m);V_m)_{m \in \mathbb{N}})$.
\end{lemma}
\begin{proof}
Let $(S^m)_{m \in \mathbb{N}}\in C^*_u((P_d(G_m); H_{d,m})_{m \in \mathbb{N}})$ and $(T^m_s)_{m\in \mathbb{N}} \in A(P_d(G_m);V_m)_{m \in \mathbb{N}})$. We shall show that $(S^mT^m_s)_m \in A(P_d(G_m);V_m)_{m \in \mathbb{N}})$. For each $m\in \mathbb{N}$, it is clear that the map $s \mapsto S^mT^m_s$ is norm continuous.

Since $\sup_{m} {\rm Prop}_{P_d(G_m)} S^m< \infty$, we have that
$$\sup_{m\in \mathbb{N},s \in [1,\infty)} {\rm Prop}_{P_d(G_m)} S^mT^m_s<\infty.$$
Therefore, Condition $(1)$ of Definition \ref{Def:TwistedRoeAlgebra} is verified. Condition $(2)$ of Definition \ref{Def:TwistedRoeAlgebra} follows from the fact that ${\rm Prop}_{V_m}(S^m)=0$ for all $m$.

Next, we shall show that $(S^mT^m_s)_m$ satisfies Condition $(3)$ in Definition \ref{Def:TwistedRoeAlgebra}.
Recall that $\chi_{V_m,R}$ and $\chi_{\varphi_m(x),R}$ are the characteristic functions of the closed ball of radius $R$ centered at $0$ and $\varphi_m(x)$ in $V_m$, respectively. On the one hand, note that
$$\|S^mT^m_s\chi_{V_m,R}^V -S^m\|\leq \|S^m\|\cdot \|T^m_s\chi_{V_m,R}^V -T^m_s\|.$$
It follow that
$$\lim\limits_{R \to \infty}\sup_{s \in [1,\infty)} \|S^mT^m_s\chi_{V_m,R}^V -S^mT^m_s\|=0.$$
We can view each element $S^m$ as a matrix $S^m=(S^m_{x,y}\otimes {\rm Id}_{\mathcal{L}^2_{V_m}})_{x,y\in X^m_d}$. Assume that
$$R\geq \sup_{m\in \mathbb{N}}\sup_{d(x,y)\leq {\rm Prop}_{P_d(G_m)}(S^m)} |\varphi_m(x)-\varphi_m(y)|.$$
Let $R_1=R- \sup_{m\in \mathbb{N}}\sup_{d(x,y)\leq {\rm Prop}_{P_d(G_m)}(S^m)} |\varphi_m(x)-\varphi_m(y)|.$
It follows from the matrix multiplication that $$
\chi_{V_m,R}^VS^mT_s=\chi_{V_m,R}^VS^m\chi_{V_m,R_1}^VT_s
$$
for all $R>0$, $m\in \mathbb{N}$ and $s \in [1,\infty)$. Therefore, by Condition $(3)$ in Definition \ref{Def:TwistedRoeAlgebra}, we have that
$$\lim\limits_{R \to \infty}\sup_{m,s} \|\chi_{V_m,R}^VS^mT^m_s -S^mT^m_s\|\leq \lim\limits_{R \to \infty}\sup_{m,s} \left(\|\chi_{V_m,R}^VS^m(\chi_{R_1}^VT^m_s-T^m_s)\|+\|S^m(\chi_{V_m,R_1}^VT^m_s -T^m_s)\|\right)=0.$$
This finishes the proof.
\end{proof}

\begin{lemma}\label{lem:T and F commutate}
For any $T=(T^m)_{m \in \mathbb{N}}\in C_u^*((P_d(G_m);H_{d,m})_{m \in \mathbb{N}})$ and $s \in [1, \infty)$, the function $s \mapsto [T, F_s]$
is in $A((P_d(G_m); V_m)_{m \in \mathbb{N}})$.
\end{lemma}

\begin{proof}
Let $T=(T^m)_{m\in \mathbb{N}}\in C^*_u((P_d(G_m);H_{d,m})_{m \in \mathbb{N}})$.
For any $\epsilon >0$ and $\Psi$ in Proposition \ref{prop:property of F}, we have that
$$
\|[T^m,F^m_s]-[T^m,F^{\Psi,m}_s]\|=\|[T^m, F^m_s-F_s^{\Psi, m}]\|\leq 2\epsilon \cdot \|T^m\|
$$
for all $m$ and $s$. It suffices to show that
$$
s \mapsto ([T^m, F^{\Psi,m}_{s}])_{m \in\mathbb{N}} \in A((P_d(G_m); V_m)_{m\in \mathbb{N}}).
$$
We have that
$$[T^m,F^{\Psi,m}_{s}]_{x,y}=T^m_{x,y}\otimes (F^{\Psi,m}_{\varphi_m(x)}-F^{\Psi,m}_{\varphi_m(y)}).$$

It follows from the part $(8)$ of Proposition \ref{prop:property of F} that for each $m$, the map $s \mapsto [T, F_s^{\Psi, m}]$ is bounded and norm continuous. Since $F^{\Psi,m}$ has $P_d(G_m)$-propagation zero and uniformly bounded $V_m$-propagation in $m \in \mathbb{N}$, the map $s \mapsto [T,F^{\Psi,m}_s]$ satisfies Condition $(1)$ and $(2)$. For Condition $(3)$ of Definition \ref{Def:TwistedRoeAlgebra}, we compute matrix coefficients
$$
((\chi_R^m)^V[T, F^{\Psi,m}_s])_{xy}=T_{xy}\otimes \chi_{\varphi_m(x),R}(F^{\Psi,m}_{s,\varphi_m(y)}-F^{\Psi, m}_{s, \varphi_m(x)}).
$$
By the part $(10)$ of Proposition \ref{prop:property of F}, we have that
$$
\lim\limits_{R \to \infty} \sup_{s\in [1, \infty)}\|[T^m,F^{\Psi,m}_s]-\chi_R^V[T^m,F^{\Psi,m}_s]\|=0.
$$
Similarly, we obtain that
$$
\lim\limits_{R \to \infty} \sup_{s\in [1, \infty)}\|[T^m,F^{\Psi,m}_s]-[T^m,F^{\Psi,m}_s]\chi_R^V\|=0.
$$
This finishes the proof.

\end{proof}

\begin{lemma}\label{lem:projections in multiplier}
Let $p=(p^m)_{m \in \mathbb{N}} \in C_u^*((P_d(G_m);H_{d, m})_{m\in \mathbb{N}})$ be a projection. Then $(pFp)^2-p$ is in the corner $pA((P_d(G_m); V_m)_{m \in\mathbb{N}})p$.
\end{lemma}
\begin{proof}
For each $m$, and each $\epsilon >0$, let $\Psi$ be the function in Proposition \ref{prop:property of F}. Since $((F^{\Psi,m}_s)^2-{\rm Id}_{H_{d,V_m}})p^m$ is close to $((F^{m}_s)^2-{\rm Id}_{H_{d,V_m}})p^m$, and $F$ and $p$ are multipliers of $A((P_d(G_m);V_m)_{m \in \mathbb{N}})$, it suffices to show that the maps $(s \mapsto (F_s^m)^2p^m-p^m)_{m \in \mathbb{N}}$ is in $A((P_d(G_m);V_m)_{m \in \mathbb{N}})$ by Lemma \ref{lem:T and F commutate}. To prove this, we consider the sequence of maps $(s \mapsto (F_s^m)^2q^m-q^m)_{m \in \mathbb{N}}$, where each $q^m$ is an approximation to $p^m$ with finite $P_d(G_m)$-propagation. It follows from the part $(7)$ of Proposition \ref{prop:property of F} that $((F^{\Psi,m}_s)^2-{\rm Id}_{H_{d,V_m}})q^m$ is norm continuous. It is routine to check Condition $(1)$-$(3)$ in Definition \ref{Def:TwistedRoeAlgebra} for these maps using Proposition \ref{prop:property of F} and finiteness of the propagation of $q^m$.
\end{proof}

Now, we are ready to define the index maps
$$
{\rm Ind}_F: K_*(C_u^*((P_d(G_m);H_{d,m})_{m \in \mathbb{N}})) \to K_*(A((P_d(G_m); V_m)_{m \in \mathbb{N}}))
$$
and
$$
{\rm Ind}^L_F: K_*(K_*(C_{u,L}^*((P_d(G_m);H_{d,m})_{m \in \mathbb{N}}))\to K_*(A((P_d(G_m). V_m)_{m \in \mathbb{N}})).
$$
\textbf{The construction of index map in $K$-theory:} Let $\mathcal{H}=\mathcal{H}^+\oplus \mathcal{H}^-$ be a graded Hilbert space with grading operator $U$ in the sense that $U$ is a self-adjoint unitary on $\mathcal{H}$ such that $\mathcal{H}^+$ and $\mathcal{H}^-$ are the eigenspaces of $\pm 1$, respectively. Let $A\subseteq B(\mathcal{H})$ be a $C^*$-subalgebra. Let $F\in B(\mathcal{H})$ be an odd operator of the form
\begin{equation*}
F=\begin{pmatrix}
0 & V \\
W & 0
\end{pmatrix}
\end{equation*}
for some operators $V: \mathcal{H}^-\to \mathcal{H}^+$ and $W: \mathcal{H}^+\to \mathcal{H}^-$. Suppose that $F$ satisfies
\begin{enumerate}[(1)]
    \item $F$ is in the multiplier algebra of $A$;
    \item $F^2-1\in A$.
\end{enumerate}
Then we can define the index class ${\rm Ind}[F]\in K_0(A)$ of $F$ to be
\begin{equation*}
    {\rm Ind}[F]=\begin{bmatrix}
(1-VW)^2 & V(1-WV) \\
W(2-VW)(1-VW) & WV(2-WV)
\end{bmatrix}
-
\begin{bmatrix}
0 & 0\\
0 & 1
\end{bmatrix}.
\end{equation*}

Let $F=(F^m)_{m \in\mathbb{N}}$ be the operator defined in Definition \ref{def:F^m_s}. For each projection $p=(p_m)_{m\in \mathbb{N}}\in C^*_u((P_d(G_m);H_{d,m})_{m \in \mathbb{N}})$, it follows from Lemma \ref{lem: F as self-adjoint, odd in multiplier}--\ref{lem:projections in multiplier} that
\begin{enumerate}[(1)]
    \item $pFp$ is an odd, self-adjoint operator in the multiplier algebra of $pA((P_d(G_m);V_m)_{m \in \mathbb{N}})p$;
    \item $(pFp)^2-p\in pA((P_d(G_m);V_m)_{m \in \mathbb{N}})p$.
\end{enumerate}
Following the above standard construction of index maps, we obtain an index class
$$
{\rm Ind}(pFp)\in K_0(pA((P_d(G_m);V_m)_{m \in \mathbb{N}})p).
$$
It is routine to verify that the map
$$
{\rm Ind}: K_0(C^*_u((P_d(G_m);H_{d,m})_{m \in \mathbb{N}}))\to K_0(pA((P_d(G_m);V_m)_{m \in \mathbb{N}})p),~~[p]\mapsto {\rm Ind}(pFp)
$$
is a homomorphism.
Composing with the map
$$
\imath_*:K_0(pA((P_d(G_m);V_m)_{m\in \mathbb{N}})p)\hookrightarrow K_0(A((P_d(G_m);V_m)_{m \in \mathbb{N}}))
$$
induced by the inclusion $\imath$ of the $C^*$-subalgebra of $pA((P_d(G_m);V_m)_{m \in \mathbb{N}})p$ into $A((P_d(G_m);V_m)_{m \in \mathbb{N}})$, we obtain a homomorphism
$$
{\rm Ind}_F:K_0(C^*_u((P_d(G_m);H_{d,m})_{m \in \mathbb{N}}))\to K_0(A((P_d(G_m);V_m)_{m \in \mathbb{N}})),~~[p]\mapsto \imath_*\circ {\rm Ind}(pFp).
$$
This map is called the index map.

Passing to suspension and applying the above construction pointwise, we can also define the index maps on the level of $K_1$-groups of twisted Roe algebras and the $K$-theory of the twisted localization algebras. Now we are ready to prove the main result in this section.
\begin{proposition}\label{prop:index map iso}
Let $d>0$. For any fixed $s\in [1,\infty)$, the compositions
$$
K_*(C^*_u(P_d(G_m);H_{d,m})_{m\in \mathbb{N}})) \xrightarrow{{\rm Ind}_F} K_*(A((P_d(G_m); V_m)_{m\in \mathbb{N}}))
\xrightarrow{\imath^s_*}K_*(C^*_u((P_d(G_m);H_{d,V_m})_{m\in \mathbb{N}}))
$$
and
$$
K_*(C^*_{u,L}((P_d(G_m);H_{d,m})_{m\in \mathbb{N}})) \xrightarrow{{\rm Ind}^L_F} K_*(A_L((P_d(G_m);V_m)_{m\in \mathbb{N}}))
\xrightarrow{\imath^s_{L,*}}K_*(C_{u,L}(P_d(G_m);H_{d,V_m})_{m\in \mathbb{N}})
$$
are isomorphisms.
\end{proposition}

\begin{proof}
We will focus on the case of twisted Roe algebras and the localized one follows from the same arguments.

For simplicity of notations, we still denote the compositions $\imath^s_*\circ {\rm Ind}_F$ and $\imath^s_{L,*}\circ {\rm Ind}^L_F$ as ${\rm Ind}_F$ and ${\rm Ind}^L_F$, respectively. For each $m\in \mathbb{N}$, define a map $\kappa_m:V_m\to V_m$ by
\begin{equation*}
  \kappa_m(x)=\begin{cases}
        \frac{x}{\|x\|}(\|x\|-1)  & \|x\|\geq 1\\
        0                                     &\|x\|<1.
  \end{cases}
\end{equation*}
For all $m, n \in \mathbb{N}$, we consider the following maps:
$$
F^{m,\infty}_{s}: H_{d,V_m}\to H_{d,V_m},\quad \delta_x\otimes \xi\otimes u\mapsto \delta_x\otimes \xi\otimes F^{m}_{s+2d_m,0}(u),
$$
$$
F^{m,n}_{s}: H_{d,V_m}\to H_{d,V_m},\quad \delta_x\otimes \xi\otimes u\mapsto \delta_x\otimes \xi\otimes F^{m}_{s+2d_m,\kappa_m^n(\varphi_m(x))}(u),
$$
and
$$
F^{m,0}_{s}: H_{d,V_m}\to H_{d,V_m},\quad \delta_x\otimes \xi\otimes u\mapsto \delta_x\otimes \xi\otimes F^{m}_{s+2d_m,\varphi_m(x)}(u).
$$
For simplicity, we denote
$$
F^{(n)}=(F^{m,n}_{s})_{m \in\mathbb{N}},~~~F^{(\infty)}=(F^{m,\infty}_{s+2d_m})_{m \in\mathbb{N}},~~~\mbox{and}~~~F^{(0)}=(F^{m,0}_{s+2d_m})_{m \in\mathbb{N}}.
$$
Note that all the operators $F^{(n)}$ for $n \in \mathbb{N}\cup \{0,\infty\}$ are multipliers of $C^*_u(P_d(G_m);H_{d, V_m})$.

%
By Lemma \ref{lem: F as self-adjoint, odd in multiplier}--\ref{lem:projections in multiplier}, for each $n\in\mathbb{N}\cup\{0,\infty\}$, and each projection $p=(p_m)_{m \in\mathbb{N}}\in C^*_u((P_d(G_m);H_{d,m})_{m \in\mathbb{N}})$, we have that
\begin{itemize}
    \item $pF^{(n)}p$ is an odd, self-adjoint operator in the multiplier algebra of $pA((P_d(G_m);V_m)_{m \in \mathbb{N}})p$;
    \item $(pF^{(n)}p)^2-p\in pA((P_d(G_m);V_m)_{m \in \mathbb{N}})p$.
\end{itemize}
Following the construction of index maps, we obtain an index map for each $n\in \mathbb{N}\cup\{0,\infty\}$. Composing the obtained map with the map $\imath^s_*$ induced by the evaluation map, we get a homomorphism
$$
{\rm Ind}_{F^{(n)}}:K_*(C^*_u((P_d(G_m);H_{d,m})_{m \in\mathbb{N}}))\to
K_*(C^*_u((P_d(G_m);H_{d,V_m})_{m \in\mathbb{N}})).
$$
for each $n\in\mathbb{N}\cup\{0,\infty\}$. To prove ${\rm Ind}_{F^{(0)}}={\rm Ind}_{F}$ is an isomorphism, we will first show that ${\rm Ind}_{F^{(\infty)}}={\rm Ind}_{F^{(0)}}$, then we show that ${\rm Ind}_{F^{(\infty)}}$ is an isomorphism.

Now, let us show that ${\rm Ind}_{F^{(\infty)}}={\rm Ind}_{F^{(0)}}$. It suffices to show that ${\rm Ind}_{F^{(\infty)}}[p]={\rm Ind}_{F^{(0)}}[p]$ for each projection $p=(p_m)_{m \in\mathbb{N}}\in C^*_u((P_d(G_m);H_{d,V_m})_{m \in\mathbb{N}})$.

For each $m\in \mathbb{N}$, denote by $(\mathcal{L}^2_{V_m})^{\oplus \infty}$ the direct sum of infinity many copies of $\mathcal{L}^2_{V_m}$.
Define
$$H^{\infty}_{d,V_m}=\ell^2(X^m_d)\otimes H\otimes (\mathcal{L}^2_{V_m})^{\oplus \infty}.$$
Let $C^*_u((P_d(G_m);H^{\infty}_{d,V_m})_{m \in \mathbb{N}})$ be the Roe algebra of $(P_d(G_m))_{m \in \mathbb{N}}$ with respect to the modules $(H^{\infty}_{d,V_m})_{m \in \mathbb{N}}$. Define an inclusion map
$$\imath: C^*_u((P_d(G_m);H_{d,V_m})_{m \in \mathbb{N}})\to C^*_u((P_d(G_m);H^{\infty}_{d,V_m})_{m \in \mathbb{N}}),~~~(T_m)_{m\in \mathbb{N}}\mapsto ({\rm diag}(T_m,0,0\cdots))_{m\in \mathbb{N}}.$$
It is obvious that the map $\imath$ is induced by the sequence of isometries covering the identity maps of the sequence $(P_d(G_m))_{m \in \mathbb{N}}$, thus it induces an isomorphism on $K$-theory.

For each $n$ and each projection $p\in C^*_u((P_d(G_m);H_{d,m})_{m \in\mathbb{N}})$, the $K$-theory class ${\rm Ind}_{F^{(n)}}[p]$ can be represented by a difference of projections
$${\rm Ind}_{F^{(n)}}[p]=[(p^{(n)}_m)_{m\in\mathbb{N}}]-[(q^{(n)}_m)_{m\in \mathbb{N}}],$$
where $(p^{(n)}_m)_{m\in \mathbb{N}}\in M_2(pC^*_u(P_d(G_m);H_{d,V_m})_{m \in \mathbb{N}})p)$.
By the construction of the index maps, each projection $(q^{(n)}_m)_{m \in\mathbb{N}}$ consists a sequence of scalar matrix thus does not depends on the $F^{(n)}$.

Now, we consider the projections
$$
{\rm diag}((p^{(0)}_m)_m,(p^{(1)}_m)_m ,,(p^{(2)}_m)_m,\cdots), \mbox{~~and~~}\quad {\rm diag}((p^{(\infty)}_m)_m,(p^{(\infty)}_m)_m ,,(p^{(\infty)}_m)_m,\cdots)
$$
in the multiplier algebra of $M_2(C^*_u((P_d(G_m);H_{d,V_m}^{\infty})_{m \in \mathbb{N}}))$. For every element $x \in V_m$, there exists a suitably large $n$ such that $\kappa^n_m(x)=0$. As a result, the difference of the above two element is an element in $M_2(C^*_u((P_d(G_m);H_{d,V_m}^{\infty})_{m \in \mathbb{N}}))$. Therefore,
$$
a=\left[{\rm diag}((p^{(0)}_m)_m,(p^{(1)}_m)_m ,(p^{(2)}_m)_m,\cdots)\right]-\left[{\rm diag}((p^{(\infty)}_m)_m,(p^{(\infty)}_m)_m, (p^{(\infty)}_m)_m,\cdots)\right]
$$
defines an element in $K_*(C^*_u((P_d(G_m);H^{\infty}_{d,V_m})_{m \in \mathbb{N}}))$.
On the other hand, we can consider the following $K$-theory class
$$
b=\left[{\rm diag}((p^{(0)}_m)_m,0,0,\cdots)\right]-\left[{\rm diag}((p^{(\infty)}_m)_m,0,0,\cdots)\right],
$$
in $K_0(C^*_u(P_d(G_m);H^{\infty}_{d,V_m})_{m \in \mathbb{N}})$. We claim that $a+b=a$, whence $b=0$. Consider the path $(F^{(n),r})_{r \in [0,1]}$ of operators on $H_{d,V_m}$ defined by
$$
F_s^{m,n,r}: \delta_x\otimes \xi\otimes u \mapsto \delta_x\otimes \xi\otimes F_{s+2d_m,(1-r)\kappa^n_m(\varphi_m(x))+r\kappa^{n+1}_m(\varphi_m(x))}u.
$$
By Proposition \ref{prop:property of F}, we have that
$$
\|F_s^{m,n,r}-F_s^{m,n,r'}\|\leq 3|r-r'|
$$
for all $r,r'\in [0,1]$ and $n,m\in \mathbb{N}$ and $s \in [1,\infty)$. Thus $({\rm Ind}_{F^{(n),r}}[p])=[(p^{(n),r}_m)_{m}]-[(q_m)_m]$ where $(p^{(n),r}_m)_m$ is a uniformly continuous path connecting $(p^{(n)})_m$ and $(p^{(n+1)})_m$, and $(q_m)_m$ is the sequence of scalr metrices in the definition of index maps. As a result, we have that
$$
a=\left[{\rm diag}((p^{(1)}_m)_m,(p^{(2)}_m)_m ,(p^{(3)}_m)_m,\cdots)\right]-\left[{\rm diag}((p^{(\infty)}_m)_m,(p^{(\infty)}_m)_m, (p^{(\infty)}_m)_m,\cdots)\right].
$$
In addition, we have that
$$
a=\left[{\rm diag}(0,(p^{(1)}_m)_m,(p^{(2)}_m)_m ,(p^{(3)}_m)_m,\cdots)\right]-\left[{\rm diag}(0,(p^{(\infty)}_m)_m,(p^{(\infty)}_m)_m, (p^{(\infty)}_m)_m,\cdots)\right].
$$
Thus, we have that $a+b=a$. It follows that $b=0$. Therefore, ${\rm Ind}_{F^{(0)}}[p]={\rm Ind}_{F^{(\infty)}}[p]$.


Next, we will conclude the proof by showing that ${\rm Ind}_{F^{(\infty)}}={\rm Id}$. For each $m$, let $p^m_0$ be the projection onto the one-dimensional kernel of $F^m_{s,0}=f((s+2d_m)^{-1}D+C)$, where $f=x(1+x^2)^{-1/2}$.
For all $m \in \mathbb{N}$ and $r\in [1,\infty)$, consider the path
$$
E^m_r=  f(r((s+2d_m)^{-1})D+C))
$$
When $r=\infty$, each $E^m_{\infty}$ is a constant path of operators which decomposes with respect to the grading as
\begin{equation*}
E^m_{\infty}=\begin{bmatrix}
0 & 1\\
1-p_0^m & 0
\end{bmatrix}.
\end{equation*}
It is well-known that the spectrum of $(B^m_{s+2d_m,0})^2$ is the set $\{2n/\sqrt{s+2d_m}:n\in \mathbb{Z}, n\geq 0\}$ for each $m$ and $s$. As a result, the tuple of operators $(B^m_{s+2d_m,0})_{m \in \mathbb{N}} $ might have no spectrum gap. Thus $((E^m_r)_{m\in \mathbb{N}})_{r\in [0,1]}$  is not necessary a continuous path in $r$, despite each path $(E^{m}_r)_{r\in [0,1]}$ is continuous for $m \in \mathbb{N}$.
Denote by $E_r=((E^m_r)_{r \in [1, \infty)})_{m \in \mathbb{N}}$. Following the previous construction of index maps, we obtain an index map
$$
{\rm Ind}_{E_r}: K_*(C^*_u((P_d(G_m);H_{d,m})_{m \in\mathbb{N}}))\to
K_*(C^*_u((P_d(G_m);H_{d,V_m})_{m \in\mathbb{N}})).
$$
for each $r \in [1,\infty)\cup\{\infty\}$.
By the definition of index maps, we have that
$${\rm Ind}_{E_{\infty}}[(p^m)_{m\in \mathbb{N}}]=[(p^m\otimes p^m_0)]$$
for all $[(p^m)_{m\in \mathbb{N}}] \in K_*(C_u^*(P_d(G_m);H_{d,m})_{m \in \mathbb{N}})$. Consider the isometry
$$V^m: \ell^2(X^m_d)\otimes H\to \ell^2(X^m_d)\otimes H\otimes \mathcal{L}^2_{V_m},~~~~\delta_x\otimes v \mapsto \delta_x\otimes v\otimes v^m_0$$
where $v^m_0\in \mathcal{L}^2_{V_m}$ is the unit vector in the kernel of the operator $B^m_{s+2d_m,0}$. It is obvious that the sequence of isometries $(V^m_{m \in \mathbb{N}})$ covers the identity maps on the sequence $(P_d(G_m))_{m \in \mathbb{N}}$. Moreover, the map ${\rm Ind}_{E_{\infty}}$ is induced by the isometries $(V^m)_{m \in \mathbb{N}}$ on $K$-theory, therefore ${\rm Ind}_{E_{\infty}}={\rm Id}$.

Next, we shall show that ${\rm Ind}_{E_1}={\rm Ind}_{E_{\infty}}$ using the same arguments in the previous claim. It suffices to prove that ${\rm Ind}_{E_1}([p])={\rm Ind}_{E_{\infty}}([p])$ for all projection $p\in C^*_u((P_d(G_m);H_{d,m})_{m \in\mathbb{N}})$. For each projection  $p\in C^*_u((P_d(G_m);H_{d,m})_{m \in \mathbb{N}})$ and each $r \in [1,\infty)$, we can express
$${\rm Ind}_{E_r}[p]=[(p^m_r)_m]-[(q^m)_{m \in \mathbb{N}}].$$
Again, since each $q^m$ is a scalar projections, we only need to show that
$$[(p^m_1)_m]=[(p^m_{\infty})_m]$$
in $K_0(C^*_u((P_d(G_m);H_{d,V_m})_{m \in \mathbb{N}}))$. Recall that we considered the map
$$\imath: C^*_u((P_d(G_m);H_{d,V_m})_{m \in \mathbb{N}})\to C^*_u((P_d(G_m);H^{\infty}_{d,V_m})_{m \in \mathbb{N}}),~~~(T_m)_{m\in \mathbb{N}}\mapsto ({\rm diag}(T_m,0,0\cdots))_{m\in \mathbb{N}}.$$
Since this map induces an isomorphism on $K$-theory, we need to show that $$\imath_*([(p^m_1)_m])=\imath_*([(p^m_{\infty})_m])$$

Note that for each $m\in \mathbb{N}$, $(E^m_r)_{r \in [1,\infty)}$ is a norm continuous path. By the definition of index maps, each path of projections $(p_{r}^m)_{r \in [1,\infty)}$ is uniformly continuous in $r$ in $M_2(C^*(P_d(G_m);H_{d,V_m})$. Therefore, there exist finitely many constant $1=r_1<r_2<r_2<\cdots<r_{k_m}$ for each $m\in \mathbb{N}$ such that
\begin{equation}\label{inequality}
\|p^{m}_{r_i}-p^m_{r_{i+1}}\|<1/2
\end{equation}
for all $i=1,2\cdots, k_m-1$. For brevity, we set $P^m_i=p^m_{r_i}$ for all $i=1,2\cdots, k_m$, $P^m_{i}=p^m_{r_{k_{m}}}$ for all $i\geq k_m$, and $P^m_{\infty}=p^m_{\infty}$ for all $m \in \mathbb{N}$. Consider the projections
$$
{\rm diag}((P^m_1)_m,(P^m_2)_m,\cdots), ~~~~~~\mbox{and}~~~~ {\rm diag}((P^m_{\infty})_m,(P_{\infty}^m)_m,\cdots).
$$
By definitions, their difference is in the Roe algebra  $C^*_u(P_d(G_m);H^{\infty}_{d,V_m})_{m \in\mathbb{N}}$.
Thus we obtain the following $K$-theory class in  $K_*(C^*_u(P_d(G_m);H^{\infty}_{d,V_m})_{m \in\mathbb{N}})$:
$$
c=\left[{\rm diag}((P^m_1)_m,(P^m_2)_m,\cdots)\right]-\left[{\rm diag}((P^m_{\infty})_m,(P_{\infty}^m)_m,\cdots)\right].
$$
Denote
$$
d=\left[{\rm diag}((P^m_1)_m,0,0,\cdots)\right]-\left[{\rm diag}((P^m_{\infty})_m,0,0,\cdots)\right].
$$
To prove $\imath_*([(p^m_1)_m])=\imath_*([(p^m_{\infty})_m])$, it suffices to show that $d=0$.
By the inequality (\ref{inequality}) above, we have that
$$
\|{\rm diag}(P^m_1,P^m_2,\cdots)-{\rm diag}(P^m_2,P^m_3,\cdots)\|< 1/2
$$
for all $m\in\mathbb{N}$. As a result we have that
$$
\|{\rm diag}((P^m_1)_{m \in \mathbb{N}},(P^m_2)_{m \in \mathbb{N}},\cdots)-{\rm diag}((P^m_2)_{m \in \mathbb{N}},(P^m_3)_{m \in \mathbb{N}},\cdots)\|< 1/2
$$
Recall that in a $C^*$-algebra $\mathcal{A}$, for any two projections $p, q\in \mathcal{A}$, if $\|p-q\|<1/2$, then $[p]=[q]\in K_0(\mathcal{A})$. Accordingly, we have that
\begin{align*}
  c= & \left[{\rm diag}((P^m_1)_{m},(P^m_2)_{m},\cdots)\right]-\left[{\rm diag}((P^m_{\infty})_m,(P_{\infty}^m)_m,\cdots)\right]\\
= & \left[{\rm diag}((P^m_2)_{m},(P^m_3)_{m},\cdots)\right]-\left[{\rm diag}((P^m_{\infty})_m,(P_{\infty}^m)_m,\cdots)\right]\\
= & \left[{\rm diag}(0,(P^m_2)_{m},(P^m_3)_{m},\cdots)\right]-\left[{\rm diag}(0,(P^m_{\infty})_m,(P_{\infty}^m)_m,\cdots)\right]+\left[{\rm diag}(P^m_1)_m,0,\cdots)]+[{\rm diag}(P^m_{\infty})_m,0,\cdots)\right]\\
= & c+d
\end{align*}
Therefore, we have that $d=0$. Thus, ${\rm Ind}_{E_1}={\rm Ind}_{E_{\infty}}$. This finishes the proof.
\end{proof}

\section{The twisted coarse Baum--Connes conjecture}\label{Sec: twisted Baum-Connes}

In this section, we shall first study the $K$-theory of the twisted Roe algebras and localization algebras. Then we prove the twisted coarse Baum--Connes conjecture using the cutting-and-pasting method.

The following is the main result of this section:
\begin{proposition}\label{prop:twistedCBC}
Let $(1 \to N_m\to G_m \to Q_m\to 1)_{m \in \mathbb{N}}$ be a sequence of short exact sequences of finite groups such that $(N_m)_m$, $(G_m)_m$ and $(Q_m)_m$ have bounded geometry. Assume that the sequences $(N_m)_m$ and $(Q_m)_{m\in\mathbb{N}}$ are uniformly coarsely embeddable into Hilbert space. If $(\varphi_m: Q_m\to V_m)_{m \in\mathbb{N}}$ is the sequence of coarse embeddings where each $V_m$ is an Euclidean space of  even-dimension, then the map
$$
e_*:\lim\limits_{d \to \infty}K_*( A_L(P_d(G_m); V_m)_{m \in \mathbb{N}})\to \lim\limits_{d \to \infty} K_*(A(P_d(G_m); V_m)_{m \in \mathbb{N}})
$$
induced by the evaluation map $e$ on $K$-theory is an isomorphism.
\end{proposition}

To prove this result, we need to discuss the $K$-theory of those twisted algebras associated with closed subsets of $V_m$. Let $B=(B_m)_{m \in \mathbb{N}}$ be a sequence of subsets such that each $B_m$ is a closed subset of $V_m$. Recall that
$$
A((P_d(G_m); B_m)_{m \in \mathbb{N}})=(1\otimes \chi_B)A((P_d(G_m); V_m)_{m \in \mathbb{N}})(1\otimes \chi_B)
$$
and
$$
A_L((P_d(G_m); B_m)_{m \in \mathbb{N}})=(1\otimes \chi_B)A((P_d(G_m); V_m)_{m \in \mathbb{N}})(1\otimes \chi_B).
$$

In addition, we assume that for each $m\in \mathbb{N}$ there exists a constant $r>0$ and a finite subset $\{x^i_m\in G_m: 1\leq i \leq k_m\}$ such that
\begin{equation}\label{condition:decomp}
B_m\subseteq\bigsqcup^{k_m}_{i=1}B_r(\varphi_m(x^i_mN_m), V_m),
\end{equation}
where $B_r(\varphi_m(x^i_mN_m), V_m)=\{v\in V_m:\|v-\varphi_m(x^i_mN_m)\|\leq r\}$.

\begin{proposition}\label{prop:local-iso}
Let $B=(B_m)_{m \in \mathbb{N}}$ be a sequence of closed subsets satisfying the above condition (\ref{condition:decomp}). If the sequence $(N_m)_{m \in \mathbb{N}}$ is coarsely embeddable into Hilbert space, then the map
$$
e_*:\lim\limits_{d \to \infty}K_*(A_L((P_d(G_m);B_m)_{m \in\mathbb{N}}) \to \lim\limits_{d \to \infty}K_*(A((P_d(G_m);B_m)_{m \in \mathbb{N}}))
$$
induced by the evaluation map on $K$-theory is an isomorphism.

\end{proposition}
We shall prove this proposition later. Now let us use it to prove Proposition \ref{prop:twistedCBC}.

\begin{proof}[Proof of Proposition \ref{prop:twistedCBC}]
For each $r>0$ and for each $m$, let $B_{m,r}=\cup_{x \in Q_m}B_r(\varphi_m(x),V_m)\subseteq V_m$,
where $(\varphi_m:Q_m \rightarrow V_m)_{m \in\mathbb{N}}$ is the sequence of coarse embeddings and
$$B_r(\varphi_m(x),V_m)=\{h\in V_m:\|h-\varphi_m(x)\|\leq r\}.$$
For any fixed $r>0$, denote by $(B_{m,r})_{m\in\mathbb{N}}$ the sequence of closed subsets. By the condition $(3)$ of Definition \ref{Def:TwistedRoeAlgebra}, we have
$$A((P_d(G_m); V_m)_{m\in\mathbb{N}})=\lim\limits_{r\to \infty}A((P_d(G_m); B_{m,r})_{m\in \mathbb{N}}),$$
and
$$
A_{L}^*((P_d(G_m); V_m)_{m\in \mathbb{N}})=\lim\limits_{r\to \infty}A_{L}^*((P_d(G_m); B_{m,r})_{m\in \mathbb{N}}).
$$
So, it suffices to show that, for any $r>0$ the map
$$e_{*}:\lim\limits_{d \to \infty}A_L((P_d(G_m); B_{m,r})_{m \in \mathbb{N}})=\lim\limits_{d\to \infty}A((P_d(G_m); B_{m,r})_{m \in \mathbb{N}})$$
induced by the evaluation-at-one map on $K$-theory is an isomorphism.

Since the sequence $(Q_m)_{m \in \mathbb{N}}$ has bounded geometry property and $(\varphi_m:Q_m \to V_m)_{m \in \mathbb{N}}$ is the coarse embedding, there exists a positive integer $n_r>0$ such that
\begin{enumerate}[(1)]
    \item each $Q_m$ is a disjoint union $Q_m=\bigsqcup_{i=1}^{n_r} \Lambda_{m,i}$;
    \item for each $i$ and each $m$, $d(\varphi_m(g),\varphi_m(g')))>2r$ for different elements $g$ and $g'$ in $\Lambda_{m,i}$.
\end{enumerate}

Set $B_{m,r,i}=\bigsqcup_{g\in \Lambda_{m,i}} B_r(\varphi_m(g); V_m)$. We obtain a sequence of closed sets $\left(B_{m,r,i}\right)_{m\in \mathbb{N}}$ for each $r$ and each $i$. By Proposition \ref{prop:local-iso}, we have that the map
$$e_{*}:\lim\limits_{d\to \infty}K_*(A_{L}^*((P_d(G_m); B_{m,r,i})_{m, i})) \rightarrow \lim\limits_{d\to \infty}K_*(A((P_d(G_m); B_{m,r,i})_{m, i}))$$
induced by the evaluation map on $K$-theory is an isomorphism for each $r>0$ and $1\leq i\leq n_r$. By the Mayer--Vietoris sequences in Proposition \ref{prop:SixTermESTwisted}, Proposition \ref{prop:SixTermESLocalization} and the five lemma, the map
$$e_{*}:\lim\limits_{d \to \infty}A_L((P_d(G_m);B_{m,r})_{m\in \mathbb{N}})=\lim\limits_{d\to \infty}A^*((P_d(G_m);B_{m,r})_{m\in \mathbb{N}})$$
induced by the evaluation-at-zero map on $K$-theory is an isomorphism.

Passing to infinity, we have that the map
$$e_{*}:\lim\limits_{d\to \infty}K_*(A_{L}^*((P_d(G_m); V_m)_{m \in \mathbb{N}})) \to \lim\limits_{d\to \infty}K_*(A^*((P_d(G_m); V_m)_{m \in \mathbb{N}}))$$
induced by the evaluation-at-one map on $K$-theory is an isomorphism. This finishes the proof of Proposition \ref{prop:twistedCBC}.
\end{proof}

\noindent{\textbf{Sketch of the proof of Proposition \ref{prop:local-iso}:}} We conclude this section by giving a sketch of the proof of Proposition \ref{prop:local-iso} using the coarse embeddability of the sequence $(N_m)_{m \in\mathbb{N}}$.
To prove Proposition \ref{prop:local-iso}, we analyze the $K$-theory of the $C^*$-algebras $A((P_d(G_m); B_m)_{m \in \mathbb{N}})$ and $A_L((P_d(G_m); B_m)_{m \in \mathbb{N}})$. We shall first connect these two $C^*$-algebra with some twisted Roe algebras and localization algebras of some sequence $(gN_m:g \in G_m)_{g,m}$ using the coarse embeddability of $(N_m)_{m}$. Proposition \ref{prop:local-iso} is then a result of some twisted Baum--Connes conjecture for some sequence of cosets $(gN_m)_{g,m}$ which can be proved using the cutting-and-pasting method.

Let $r>0$ be a fixed positive constant in Proposition \ref{prop:local-iso}. Recall   that in previous section we chose a sequence of closed subsets $(B_m)_{m \in \mathbb{N}}$ such that each $B_m$ is a disjoint union
$$B_m\subseteq\bigsqcup^{k_m}_{i=1}B_r(\varphi_m(x^i_mN_m), V_m),
$$
where each $\{x^i_m\in G_m: 1\leq i \leq k_m\}$ is a subset of $G_m$. Moreover, we assume that $B_{m}=\bigsqcup_{i=1}^{k_m} B_{m.i}$ where each $B_{m,i}\subseteq B_r(\varphi_m(x_m^iN_m), V_m)$ is a nonempty closed subset for all $1\leq i\leq k_m$ and $m \in \mathbb{N}$.

Denote by $X^R_{m,i}=B_R(x_m^iN_m, G_m)$ and $X^{d,R}_{m,i}=P_d(X^R_{m,i})\cap X^m_d$, where $X_d^m$ is the countable dense subset of $P_d(G_m)$ in the definition of Roe algebras of the space $P_d(G_m)$. Note that each $X^{d,R}_{m,i}$ is a countable dense subset of $P_d(X^R_{m,i})$. Consider the ample module
$$H_{d,B_{m,i}}=\ell^2(X^{d,R}_{m,i})\otimes H\otimes \mathcal{L}^2_{B_{m,i}},$$
where $\mathcal{L}^2_{B_{m,i}}$ is the Hilbert space of square-integrable functions from
$B_{m,i}$ to the complxified Clifford algebra ${\rm Cliff}_{\mathbb{C}}(V_m)$. Note that this Hilbert space is an ample $P_d(X^R_{m,i})$-module and an ample $B_{m,i}$-module, respectively. Denote by $C^*(P_d(X^R_{m,i});H_{d, B_{m,i}})$ the algebra of the $P_d(X^R_{m,i})$ with respect to the module $H_{d,B_{m,i}}$.  In fact, the Hilbert space $H_{d,B_{m,i}}$ is a subspace of $H_{d,V_m}$ by $H_{d,B_{m,i}}=(\chi_{P_d(X^R_{m,i})} \otimes {\rm Id}_H\otimes \chi_{B_{m,i}}) H_{d,V_m}$.

For simplicity, we denote by $I$ the set of all indices $(m,i)$, $I=\{(m,i):m \in \mathbb{N}, 1\leq i\leq k_m\}$, and we will use both the Greek letter $\lambda$ and $(m,i)$ to denote an element in $I$ in the rest of this section.

We define the $C^*$-algebra $A(P_d(X^R_{\lambda});B_{\lambda})_{\lambda\in I})$ is the completion of the $*$-subalgebra of $B(\oplus_{\lambda\in I} H_{d,B_{\lambda}})$ consisting of all element $(T^{\lambda}_s)_{\lambda\in I, s\in [1,\infty)}=(T^{\lambda}_{s,x,y})_{\lambda\in I,s\in [1,\infty),x,y \in X^{d,R}_{\lambda}}$ satisfying
\begin{itemize}
    \item $\sup\limits_{\lambda\in I, s\in [1,\infty)} {\rm Prop}_{P_d(X^R_{\lambda})} (T_s^{\lambda})< \infty$;
    \item $\lim\limits_{s \to \infty}\sup_{\lambda\in I} {\rm Prop}_{B_{\lambda}} (T^{\lambda}_s)=0$.
\end{itemize}
Note that $A(P_d(X^R_{\lambda});B_{\lambda})_{\lambda\in I})$ is a $C^*$-subalgebra of the twisted Roe algebra $A((P_d(G_m);H_{d,V_m})_{m \in\mathbb{N}})$. Indeed, for each element $(T^{m,i}_s)_{(m,i)\in I}\in A((P_d(X^R_{m,i});B_{m,i})_{(m,i)\in I})=A((P_d(X^R_{\lambda});B_{\lambda})_{\lambda\in I})$, define
$$T^m_s=\oplus_{i=1}^{k_m} T^{m,i}_s$$
for all $m\in \mathbb{N}$. Then $(T^m_s)_{m \in\mathbb{N}}$ is an element in $A((P_d(G_m);V_m)_{m \in\mathbb{N}})$.

The localization algebra $A_L(P_d(X^R_{\lambda});B_{\lambda})_{\lambda\in I})$ is the completion of all uniformly continuous and uniformly bounded function $T:[1,\infty)\to A(P_d(X^R_{\lambda});B_{\lambda})_{\lambda\in I})$ with $\lim\limits_{t\to \infty} {\rm Prop}_{P_d(X^R_{\lambda})}(T(t))=0$ under the norm
$$\|(T(t))_{t\in[1,\infty)}\|=\sup_{t\in [1 \infty)}\|T(t)\|.$$
We also have that the localization algebra $A_L(P_d(X^R_{\lambda});B_{\lambda})_{\lambda\in I})$ is a $C^*$-subalgebra of the twisted localization algebra $A_L((P_d(G_m);H_{d,V_m})_{m \in\mathbb{N}})$.

Comparing with Definition \ref{Def:TwistedRoeAlgebra}, the above definition of the Roe algebra $A(P_d(X^R_{\lambda});B_{\lambda})_{\lambda\in I})$ satisfies Condition $(3)$ in Definition \ref{Def:TwistedRoeAlgebra} since $B_{m,i} \subseteq B_r(\varphi_m(x_m), V_m)$ for all $(m,i)\in I$. The following result is a consequence of Condition $(3)$ in Definition \ref{Def:TwistedRoeAlgebra}.
\begin{lemma}
We have that
$$
A((P_d(G_m); V_m)_{m \in \mathbb{N}})\cong \lim\limits_{R \to \infty} A(P_d(X^R_{\lambda});B_{\lambda})_{\lambda\in I}),
$$
and
$$
A_L((P_d(G_m); V_m)_{m \in \mathbb{N}})\cong \lim\limits_{R \to \infty} A_L(P_d(X^R_{\lambda});B_{\lambda})_{\lambda\in I}).
$$
\end{lemma}

It is easy to verify that the following diagram is commutative:
\begin{equation*}
\begin{tikzcd}
A_L(P_d(X^R_{\lambda});B_{\lambda})_{\lambda\in I})\arrow[r,"e"] \arrow[d]&A(P_d(X^R_{\lambda});B_{\lambda})_{\lambda\in I})\arrow[d]\\
A_L((P_d(G_m); V_m)_{m \in \mathbb{N}})\arrow[r,"e"]& A((P_d(G_m); V_m)_{m \in \mathbb{N}})\\
\end{tikzcd}
\end{equation*}
where the horizontal maps are defined by evaluate the path at one. Moreover, we have
$$
\lim\limits_{d \to \infty}\lim\limits_{R\to \infty} A((P_d(X^R_{\lambda});B_{\lambda})_{\lambda\in I})=\lim\limits_{R\to \infty}\lim\limits_{d \to \infty} A((P_d(X^R_{\lambda});B_{\lambda})_{\lambda\in I}),
$$
and
$$
\lim\limits_{d \to \infty}\lim\limits_{R\to \infty} A_L((P_d(X^R_{\lambda});B_{\lambda})_{\lambda\in I})=\lim\limits_{R\to \infty}\lim\limits_{d \to \infty} A_L((P_d(X^R_{\lambda});B_{\lambda})_{\lambda\in I}).
$$
Therefore, to prove Proposition \ref{prop:local-iso}, it suffice to show the following result:
\begin{proposition}\label{prop:local coefficients}
Let $R>0$. If the sequence $(N_m)_{m \in \mathbb{N}}$ is coarsely embeddable into Hilbert space, then the map
$$
e_*:\lim\limits_{d \to \infty}K_*(A_L(P_d(X^R_{\lambda});B_{\lambda})_{\lambda\in I}) \to \lim\limits_{d \to \infty}K_*(A(P_d(X^R_{\lambda});B_{\lambda})_{\lambda\in I})
$$
induced by the evaluation map is an isomorphism.
\end{proposition}

In the rest of this section, we shall prove the above proposition by two steps: we first formulate a twisted version of coarse Baum--Connes conjecture for the sequence $(X^{R}_{m,i})_{m,i}$ using the coarse embedding of the sequence $(N_m)_{m \in \mathbb{N}}$ and prove it by the cutting-and-pasting method, then we reduce the above result to a twisted coarse Baum--Connes conjecture for the sequence $(X^{R}_{m,i})_{m,i}$ by constructing index maps following Section \ref{Sect: Index map}.

Note that for each fixed $R>0$, each $gN_m$ is an $R$-net in the $R$-neighborhood, and the inclusion of $gN_m$ into $B_R(gN_m, G_m)$ is an isometry. It follows that that the sequence $\{x^i_m N_m\}_{(m,i)\in I}$ is coarsely equivalent to the sequence $(B_R(x^i_mN_m, G_m))_{(m,i)\in I}$. As a result, the sequence $\left(X^R_{m,i}=B_R(x^i_mN_m, G_m)\right)_{(m,i)\in I}$ is coarsely embeddable into Hilbert space if the sequence $(N_m)_{m \in \mathbb{N}}$ is coarsely embeddable.

 Since each $x^i_mN_m$ is an $R$-net in $B_R(x^i_mN_m, G_m)$ for each $(m,i)\in I$, we can define a coarsely equivalent map $f_{m,i}: B_R(x^i_mN_m, G_m) \to x^i_mN_m$ such that $d(x,f_{m,i}(x))\leq R$ for all $x \in B_R(x^i_mN_m, G_m)$.  Composing with the isometry $x^i_mN_m\to N_m$ mapping $x^i_mg$ to $g$ for each $g \in N_m$ gives us a map, still denoted by $f_{m,i}: X^R_{m,i}=B_R(x^i_mN_m, G_m)\to N_m$ for each $m$ and each $i$. Moreover, we have a sequence of coarse equivalences $(f_{m,i}:  X^R_{m,i} \to N_m)_{(m,i)\in I}$.

Let $(\psi_m: N_m\to U_m)_{m \in \mathbb{N}}$ be a sequence of coarse embeddings, where each $U_m$ is an even-dimensional Euclidean space. Composing with the coarse equivalences $(f_{m,i}:X^R_{m,i}\to N_m)_{(m,i)\in I}$, we obtain a sequence of coarse embeddings, denoted by
$$
\left(\psi_{m,i}:X_{m,i}^R\to U_m \right)_{(m,i)\in I}.
$$
Let us define
$$
H_{d,B_{m,i},U_m}=\ell^2(X^{d,R}_{m,i})\otimes H \otimes \mathcal{L}^2_{B_{m,i}}\otimes \mathcal{L}^2_{U_m}
$$
where $\mathcal{L}^2_{U_m}$ is the Hilbert space of all the square-integrable functions from $U_m$ to the Hilbert space of complexified Clifford algebra ${\rm Cliff}_{\mathbb{C}}(U_m)$, and $\mathcal{L}^2_{B_{m,i}}$ is the Hilbert space of square-integrable functions from $B_{m,i}$ to the complexified Clifford algebra ${\rm Cliff}_{\mathbb{C}}(V_m)$.

Denote $U_{\lambda}=U_{m}$ when $\lambda=(m,i)$ and denote by $(\psi_{\lambda}:X^R_{\lambda}\to U_{\lambda})_{\lambda\in I}$ the obtained sequence of coarse embeddings. For each $\lambda\in I$ and $d>0$, the Hilbert space $H_{d,B_{\lambda},U_{\lambda}}$ is an ample $P_d(X_{\lambda}^R)$-module, an ample $U_{\lambda}$-module and an ample $B_{\lambda}$-module. Therefore, for each bounded linear operator $T\in B(H_{d,B_{\lambda},U_{\lambda}})$, we can consider $P_d(X^R_{\lambda})$-propagation, $U_{\lambda}$-propagation and $B_{\lambda}$-propagation.

For each $\lambda$, an operator $T\in B(\mathcal{L}^2_{U_{\lambda}})$ gives rise to an operator $T^{U}=(T^U_{x,y})_{x,y\in X^{d,R}_{\lambda}}$ on $H_{\lambda,B_{\lambda},U_{\lambda}}$ by
\begin{equation}\label{constr:spatial operator of subgroups geometry}
T^{U}_{x,y}= \left\{
        \begin{array}{ll}
            {\rm Id}_{\ell^2(X^{d,R}_{\lambda})\otimes H \otimes \mathcal{L}^2_{B_{\lambda}}}\otimes V_{\psi_{\lambda}(x)}TV^*_{\psi_{\lambda}(x)}, & \quad y=x \\
            0 & \mbox{otherwise}.
        \end{array}
    \right.
\end{equation}
For each $R'>0$ and $\lambda\in I$, denote by $\chi_{U_{\lambda},R}$ to be the characteristic function of the closed ball $B_{R'}(0, U_{\lambda})$.

Now we are ready to define the twisted Roe algebras for the sequence $(P_d(X^R_{\lambda}))_{\lambda\in I}$ using the above sequence of coarse embeddings.

For each $R>0$, $d>0$, and $\lambda\in I$, denote by $C^*(P_d(X^R_{\lambda});H_{d,B_{\lambda},U_{\lambda}})$ the Roe algebra of $P_d(X^R_{\lambda})$ with respect to the ample $P_d(X^R_{\lambda})$-module $H_{d,B_{\lambda},U_{\lambda}}$. let $\prod_{\lambda\in I} C_b([1,\infty) \times [1, \infty), C^*(P_d(X^R_{\lambda}); H_{d,B_{\lambda},U_{\lambda}}))$ be the product $C^*$-algebra of all bounded norm-continuous maps from $[1,\infty)\times [1,\infty)$ to the Roe algebra $C^*(P_d(X_{\lambda}); H_{d,B_{\lambda},U_{\lambda}})$ equipped with the supreme norm. Each element of the product $C^*$-algebra can be expressed as a tuple $(T^{\lambda}_{s,s'})_{\lambda\in I, s, s' \in [1,\infty)}$ where for each $\lambda$, $(s,s') \mapsto T^{\lambda}_{s, s'}$ is a bounded continuous function from $[1,\infty)\times [1,\infty)$ to $C^*(P_d(X^R_{\lambda});H_{d,B_{\lambda}, U_{\lambda}})$.

\begin{definition}\label{Def:twisted Roe of subgroups}
	For each $d>0$ and $R>0$, the algebraic twisted Roe algebra $\mathbb{A}((P_d(X^R_{\lambda}); B_{\lambda}, U_{\lambda})_{\lambda\in I})$ is defined to be the $*$-subalgebra of the product $\prod_{\lambda\in I} C_b([1,\infty)\times [1,\infty), C^*(P_d(X_{\lambda}); H_{d,B_{\lambda}, U_{\lambda}}))$ consisting of elements $(T^{\lambda}_{s,s'})_{\lambda\in I, s, s' \in [1,\infty)}$ satisfying the following properties:
\begin{enumerate}[(1)]
		\item $\sup\limits_{\lambda\in I, s, s' \in [1,\infty)} {\rm Prop}_{P_d(X^R_{\lambda})}(T^{\lambda}_{s, s'})<\infty$;
		
		\item $\lim\limits_{s \to \infty}\sup\limits_{\lambda\in I, s'\in [1.\infty)}  {\rm Prop}_{B_{\lambda}}(T^{\lambda}_{s, s'})=0$;
		
		\item $\lim\limits_{s' \to \infty}\sup\limits_{\lambda\in I, s\in [1,\infty)}  {\rm Prop}_{U_{\lambda}}(T^{\lambda}_{s,s'})=0$;
	
		\item $\lim\limits_{R' \to \infty}\sup\limits_{\lambda \in I, s,s' \in[1,\infty)} \|\chi^{U}_{U_{\lambda},R'} T^{\lambda}_{s, s'}-T^{\lambda}_{s, s'}\|=\lim\limits_{R' \to \infty}\sup\limits_{\lambda\in I, s,s' \in[1,\infty)} \|T^{\lambda}_{s, s'}\chi^{U}_{U_{\lambda},R'}-T^{\lambda}_{s,s'}\|=0$.
\end{enumerate}
The twisted Roe algebra $A((P_d(X^R_{\lambda});B_{\lambda}, U_{\lambda})_{\lambda \in I})$ of the sequence $(P_d(X^R_{\lambda}))_{\lambda\in I}$ is defined to be the closure of $\mathbb{A}((P_d(X_{\lambda}); B_{\lambda}, U_{\lambda})_{\lambda \in I})$ in $\prod_{\lambda\in I} C_b([1,\infty)\times [1,\infty), C^*(P_d(X_{\lambda});H_{d,B_{\lambda}, U_{\lambda}})_{\lambda\in I})$.
\end{definition}

We can define the algebraic twisted localization algebra $\mathbb{A}_L((P_d(X_{\lambda}); B_{\lambda}, U_{\lambda}))_{\lambda\in I})$ to be the $*$-algebra of uniformly bounded and uniformly continuous functions $T:[1, \infty) \to A((P_d(X^R_{\lambda}); B_{\lambda}, U_{\lambda})_{\lambda \in I})$ such that
$$
\lim\limits_{t \to \infty} \sup_{\lambda\in I} {\rm Prop}_{P_d(X^R_\lambda)}(T^{\lambda}(t))=0.
$$
The twisted localization algebra $A_L((P_d(X^R_{\lambda}); B_{\lambda}, U_{\lambda}))_{\lambda\in I})$ is the completion of $\mathbb{A}_L((P_d(X^R_{\lambda}); B_{\lambda}, U_{\lambda}))_{\lambda\in I})$ under the norm
$$\|T\|=\sup\|T(t)\|_{A((P_d(X^R_{\lambda}); B_{\lambda}, U_{\lambda}))_{\lambda\in I})}.$$
Naturally, there is a evaluation-at-one map
$$
e: A_L((P_d(X^R_{\lambda}); B_{\lambda}, U_{\lambda}))_{\lambda\in I})\to A((P_d(X^R_{\lambda}); B_{\lambda}, U_{\lambda}))_{\lambda\in I}).
$$
\begin{proposition}\label{prop:twisted CBC for subgroups}
For each $R>0$, the map
$$
e_*:\lim\limits_{d \to \infty} K_*(A_L((P_d(X^R_{\lambda}); B_{\lambda}, U_{\lambda}))_{\lambda\in I}))\to \lim\limits_{d \to \infty}K_*(A((P_d(X^R_{\lambda}); B_{\lambda}, U_{\lambda}))_{\lambda\in I}))
$$
induced by the evaluation map on $K$-theory is an isomorphism.
\end{proposition}

The above result can be proved by the cutting-and-pasting method. We shall sketch the proof in the following.

Let us introduce some notations before the proof. For a sequence of closed subsets $D=(D_{\lambda})_{\lambda \in I}$ with $D_{\lambda}\subseteq U_{\lambda}$, we define
$$
A(P_d(X^R_{\lambda});B_{\lambda},D_{\lambda})=(1\otimes \chi_D)A(P_d(X^R_{\lambda});B_{\lambda},U_{\lambda})(1\otimes \chi_D)
$$
and
$$
A_L(P_d(X^R_{\lambda});B_{\lambda},D_{\lambda})=(1\otimes \chi_D)A_L(P_d(X^R_{\lambda});B_{\lambda},U_{\lambda})(1\otimes \chi_D),
$$
where $1\otimes \chi_D=({\rm Id}_{\ell^2(X^{d,R}_{\lambda})\otimes H\otimes \mathcal{L}^2_{B_{\lambda}}}\otimes \chi_{D_{\lambda}})_{\lambda\in I}$ and each $\chi_{D_{\lambda}}$ is the characteristic function of $D_{\lambda}$ viewed as a projection on $\mathcal{L}^2_{U_{\lambda}}$.

Let $K=(K_{\lambda})_{\lambda\in I}$ be a sequence of closed subsets with each $K_{\lambda}\subseteq P_d(X^R_{\lambda})$. Define
$$
A((K_{\lambda};B_{\lambda},D_{\lambda})_{\lambda\in I})=(\chi_K \otimes 1)A((P_d(X^R_{\lambda});B_{\lambda},U_{\lambda})_{\lambda\in I})(\chi_K\otimes 1)
$$
and
$$
A_L((K_{\lambda};B_{\lambda},D_{\lambda})_{\lambda\in I})=(\chi_K \otimes 1)A_L((P_d(X^R_{\lambda});B_{\lambda},U_{\lambda})_{\lambda\in I})(\chi_K\otimes 1)
$$
where $\chi_K\otimes 1=(\chi_{K_{\lambda}}\otimes {\rm Id}_{H\otimes \mathcal{L}^2_{B_{\lambda}}\otimes \mathcal{L}^2_{U_{\lambda}}})_{\lambda\in I}$.

The following lemma is crucial to the prove Proposition \ref{Def:twisted Roe of subgroups}.
\begin{lemma}\label{lem:subgroup local iso}
Let $r'>0$ and let $D=(D_{\lambda})_{\lambda}$ be a sequence of closed subsets with $D_{\lambda}\subseteq U_{\lambda}$ such that each $D_{\lambda}$ decomposes as a disjoint union $D_{\lambda}=\bigsqcup_{j\in J_{\lambda}} D_{\lambda,j}$ satisfying
$$
D_{\lambda,j}\subseteq B_{r'}(\psi_{\lambda}(x_{\lambda,j}), U_{\lambda})
$$
for some collection $\{x_{\lambda,j}\}_{j\in J_{\lambda}}\subseteq X^R_{\lambda}$ with $d(\psi_{\lambda}(x_{\lambda,j}),\psi_{\lambda}(x_{\lambda,j'}))>2r'$ for all $j\neq j'$. Then the map
$$
e_*:\lim\limits_{d\to \infty}K_*(A_L((P_d(X^R_{\lambda});B_{\lambda},D_{\lambda})_{\lambda})) \to \lim\limits_{d\to \infty}K_*(A((P_d(X^R_{\lambda});B_{\lambda},D_{\lambda})_{\lambda}))
$$
is an isomorphism.
\end{lemma}

To prove Lemma \ref{lem:subgroup local iso}, we need some basic facts in the $K$-theory of $C^*$-algebras. Recall that a $C^*$-algebra $\mathcal{A}$ is said to be quasi-stable if for all $n$ there exists an isometry $v$ in the multiplier algebra of the matrix algebra $M_n(\mathcal{A})$ such that $vv^*$ is the top left matrix unit $e_{11}$.
The following lemma will be useful in the proof of Lemma \ref{lem:subgroup local iso}. The proof can be found in \cite[Lemma 12.4.3]{Willett-Yu-Book}, thus omitted.
\begin{lemma}\label{lem:alg-evaluation-iso}
Let $\mathcal{A}$ be a qusi-stable $C^*$-algebras and let $C_b([1,\infty), \mathcal{A})$ be the $C^*$-algebras of all uniformly bounded and uniformly norm-continuous functions from $[1,\infty)$ to $\mathcal{A}$. Then the evaluation map
$$C_b\left([1,\infty),\mathcal{A}\right)\to  \mathcal{A},~~~~(f(t))_{t \in [1, \infty)} \mapsto f(1)$$
 induces an isomorphism on $K$-theory.
\end{lemma}

\begin{proof}[Proof of Lemma \ref{lem:subgroup local iso}]
Denote by $H_{d,B_{\lambda},D_{\lambda,j}}=({\rm Id}_{\ell^2(X^{d,R}_{\lambda})\otimes H \otimes \mathcal{L}_{B_{\lambda}}}\otimes \chi_{D_{\lambda,j}})H_{d,B_{\lambda},U_{\lambda}}$ for each $\lambda\in I$ and each $j \in J_{\lambda}$. Let $\mathbb{A}((P_d(X^R_{\lambda});B_{\lambda}, D_{\lambda,j})_{\lambda\in I,j\in J_{\lambda}})$ be the $*$-subalgebra of the product
$$\prod_{\lambda\in I, j
\in J_{\lambda}}C_b([1, \infty), C^*(P_d(X^R_{\lambda});H_{d,B_{\lambda}, D_{\lambda,j}}))$$ consisting of all tuples $(T_{s,s'}^{\lambda,j})_{\lambda\in I, j\in J_{\lambda}, s,s'\in [\infty)}$ such that \begin{itemize}
    	\item $\sup\limits_{\lambda\in I,j \in J_{\lambda}, s, s' \in [1,\infty)} {\rm Prop}_{P_d(X^R_{\lambda})}(T^{\lambda,j}_{s, s'})<\infty$;
		
		\item $\lim\limits_{s \to \infty}\sup\limits_{\lambda\in I,j \in J_{\lambda}, s'\in [1.\infty)}  {\rm Prop}_{B_{\lambda}}(T^{\lambda,j}_{s, s'})=0$;
		
		\item $\lim\limits_{s' \to \infty}\sup\limits_{\lambda\in I, s\in [1,\infty)}  {\rm Prop}_{U_{\lambda}}(T^{\lambda,j_{\lambda}}_{s,s'})=0$;
	
		\item $\lim\limits_{R' \to \infty}\sup\limits_{\lambda \in I,j \in J_{\lambda}, s,s' \in[1,\infty)} \|\chi^{U}_{U_{\lambda},R'} T^{\lambda,j}_{s, s'}-T^{\lambda,j}_{s, s'}\|=\lim\limits_{R' \to \infty}\sup\limits_{\lambda\in I,j \in J_{\lambda}, s,s' \in[1,\infty)} \|T^{\lambda,j}_{s, s'}\chi^{U}_{U_{\lambda},R'}-T^{\lambda,j}_{s,s'}\|=0$.
\end{itemize}
Let $A((P_d(X^R_{\lambda});B_{\lambda}, D_{\lambda,j})_{\lambda\in I,j\in J_{\lambda}})$ be the completion of $\mathbb{A}((P_d(X^R_{\lambda});B_{\lambda}, D_{\lambda,j})_{\lambda\in I,j\in J_{\lambda}})$ under the operator norm on $\oplus_{\lambda\in I}H_{d,B_{\lambda},U_{\lambda}}$. Similarly, we can define the localized algebra, $A_L((P_d(X^R_{\lambda});B_{\lambda}, D_{\lambda,j})_{\lambda\in I,j\in J_{\lambda}})$, to be $C^*$-algebra generated by the uniformly bounded and uniformly continuous paths $(t \mapsto T(t))$ from $[1,\infty) $ to $\mathbb{A}((P_d(X^R_{\lambda});B_{\lambda}, D_{\lambda,j})_{\lambda\in I,j\in J_{\lambda}})$ satisfying ${\rm Prop}_{P_d(X^R_{\lambda}}(T(t))\to 0$ as $t \to \infty$.

Let $R'>0$ be any positive number. Consider the sequence of closed balls $B_{R'}=(B_{R'}(x_{\lambda,j}))_{\lambda\in I, j \in J_{\lambda}}$ where each $B_{R'}(x_{\lambda,j})$ is the ball of radius $R'$ centered at $x_{\lambda,j}$ in $X^R_{\lambda}$. We define
$$
A((P_d(B_{R'}(x_{\lambda,j});B_{\lambda}, D_{\lambda,j})_{\lambda\in I,j\in J_{\lambda}})=
(\chi_{B_{R'}}\otimes 1)A((P_d(X^R_{\lambda});B_{\lambda}, D_{\lambda,j})_{\lambda\in I,j\in J_{\lambda}})(\chi_{B_{R'}}\otimes 1)
$$
and
$$
A_L((P_d(B_{R'}(x_{\lambda,j});B_{\lambda}, D_{\lambda,j})_{\lambda\in I,j\in J_{\lambda}})=
(\chi_{B_{R'}}\otimes 1)A_L((P_d(X^R_{\lambda});B_{\lambda}, D_{\lambda,j})_{\lambda\in I,j\in J_{\lambda}})(\chi_{B_{R'}}\otimes 1)
$$
where $\chi_{B_{R'}}\otimes 1=(\chi_{B_{R'}(x_{\lambda,j})}\otimes {\rm Id}_{H\otimes \mathcal{L}^2_{B_{\lambda}}\otimes \mathcal{L}^2_{U_{\lambda}}})_{\lambda\in I, j \in J_{\lambda}}$ is a sequence of projections.

By the definition of twisted Roe algebras, we have that
$$
A_L((P_d(X^R_{\lambda});B_{\lambda}, D_{\lambda,j})_{\lambda\in I,j\in J_{\lambda}})=\lim\limits_{R'\to \infty}A_L((P_d(B_{R'}(x_{\lambda,j});B_{\lambda}, D_{\lambda,j})_{\lambda\in I,j\in J_{\lambda}})
$$
and
$$
A_L((P_d(X^R_{\lambda});B_{\lambda}, D_{\lambda,j})_{\lambda\in I,j\in J_{\lambda}})=\lim\limits_{R'\to \infty}A_L((P_d(B_{R'}(x_{\lambda,j});B_{\lambda}, D_{\lambda,j})_{\lambda\in I,j\in J_{\lambda}}).
$$
It suffices to show that
$$
e_*: \lim\limits_{d \to \infty} A_L((P_d(B_{R'}(x_{\lambda,j});B_{\lambda}, D_{\lambda,j})_{\lambda\in I,j\in J_{\lambda}}) \to \lim\limits_{d \to \infty} A((P_d(B_{R'}(x_{\lambda,j});B_{\lambda}, D_{\lambda,j})_{\lambda\in I,j\in J_{\lambda}})
$$
is an isomorphism for each fixed $R'>0$.
Since the sequence of metric spaces $(X^R_{\lambda})_{\lambda \in I}$ has bounded geometry, the dimensions of the sequence of simplices $(B_{R'}(x_{\lambda,j}))_{\lambda\in I,j \in J_{\lambda}}$ are uniformly bounded.
Therefore the inclusions $(x_{\lambda,j} \to P_d(B_{R'}(x_{\lambda,j})))_{\lambda\in I,j \in J_{\lambda}}$ is a sequence of strongly Lipschitz homotopy equivalences. As a result, this sequence of inclusions induces isomorphisms on the $K$-theory of both localization algebras and the Roe algebras. It is easy to check that the following  diagram commutes:
\begin{equation*}
    \begin{tikzcd}
    K_*(A_L((P_d(B_{R'}(x_{\lambda,j});B_{\lambda}, D_{\lambda,j})_{\lambda\in I,j\in J_{\lambda}})) \arrow[r,"e_*"]&  K_*(A((P_d(B_{R'}(x_{\lambda,j});B_{\lambda}, D_{\lambda,j})_{\lambda\in I,j\in J_{\lambda}}))\\
    K_*( A_L((x_{\lambda,j};B_{\lambda}, D_{\lambda,j})_{\lambda\in I,j\in J_{\lambda}})) \ar[u,"\cong"]\arrow[r,"e_*"]& K_*(A((x_{\lambda,j};B_{\lambda}, D_{\lambda,j})_{\lambda\in I,j\in J_{\lambda}}))\ar[u,"\cong"].
    \end{tikzcd}
\end{equation*}
where the vertical maps are isomorphisms for any suitable large $d$. Thus, it suffices to show that the map
$$
e_*:K_*(A_L((x_{\lambda,j};B_{\lambda}, D_{\lambda,j})_{\lambda\in I,j\in J_{\lambda}}))\to K_*(A((x_{\lambda,j};B_{\lambda}, D_{\lambda,j})_{\lambda\in I,j\in J_{\lambda}}))
$$
induced by the evaluation map on $K$-theory is isomorphism. Since the condition that the propagation on the sequence of points $(x_{\lambda, j})_{\lambda,j\in J_{\lambda}}$ defining the localization version on the left is vacuous, the algebra on the left-hand side $A_L((x_{\lambda,j};B_{\lambda}, D_{\lambda,j})_{\lambda\in I,j\in J_{\lambda}})$ is the $C^*$-algebra of all uniformly continuous and uniformly bounded functions from $[1,\infty)$ to the algebras on the right, $A((x_{\lambda,j};B_{\lambda}, D_{\lambda,j})_{\lambda\in I,j\in J_{\lambda}})$. By Lemma \ref{lem:alg-evaluation-iso}, the evaluation-at-one map induces an isomorphism on $K$-theory. This finishes the proof.

\end{proof}
Now, we are ready to apply the Mayer--Viotoris sequence and the five lemma to prove Proposition \ref{prop:twisted CBC for subgroups}.

\begin{proof}[Proof of Proposition \ref{prop:twisted CBC for subgroups}.]
For each $r'>0$, let $B_{\lambda}(r')=\bigcup_{x \in X^R_{\lambda}} B_{r'}(\psi_{\lambda}(x), U_{\lambda})$. By Condition $(4)$ of Definition \ref{Def:twisted Roe of subgroups}, we have that
$$
A((P_d(X^R_{\lambda});B_{\lambda},U_{\lambda})_{\lambda \in I})=\lim\limits_{r'\to \infty}A((P_d(X^R_{\lambda});B_{\lambda},B_{\lambda}(r'))_{\lambda\in I})
$$
and
$$
A_L((P_d(X^R_{\lambda});B_{\lambda},U_{\lambda})_{\lambda\in I})=\lim\limits_{r'\to \infty}A_L((P_d(X^R_{\lambda});B_{\lambda},B_{\lambda}(r')))_{\lambda\in I}).
$$
Therefore, it suffices to show that
$$e_*:\lim\limits_{d\to \infty}K_*(A_L((P_d(X^R_{\lambda});B_{\lambda},B_{\lambda}(r'))_{\lambda\in I}))\to \lim\limits_{d\to \infty}K_*(A((P_d(X^R_{\lambda});B_{\lambda},B_{\lambda}(r'))_{\lambda\in I}))$$
is an isomorphism for any $r'>0$.

For any fixed $r'>0$, since the sequence $\{X^R_{\lambda}\}_{\lambda\in I}$ has bounded geometry and
 $(\psi_{\lambda}: X^R_{\lambda}\to U_{\lambda})_{\lambda\in I}$ is a sequence of coarse embeddings, there exists a positive integer $M$ such that each $X^R_{\lambda}$ decomposes as a disjoint union $X^R_{\lambda}=\bigsqcup^M_{i=1} X_{\lambda,i}$ such that
 \begin{center}
 $B_{r'}(\psi_{\lambda}(x), U_{\lambda})\cap B_{r'}(\psi_{\lambda}(x'), U_{\lambda})=\emptyset$, if $x,x'\in X_{\lambda, i}$ with $x\neq x'$, for all $1\leq i \leq M$ and $\lambda \in I$.
\end{center}
Denote by $D_{\lambda, i}(r')=\bigsqcup_{x \in X_{\lambda,i}} B_{r'}(\varphi_{\lambda}(x), U_{\lambda})$ for each $1\leq i\leq M$.
By Lemma \ref{lem:subgroup local iso}, we know that the map
$$e_*:\lim\limits_{d\to \infty}K_*(A_L((P_d(X^R_{\lambda});B_{\lambda},D_{\lambda,i}(r'))_{\lambda\in I}))\to \lim\limits_{d\to \infty}K_*(A((P_d(X^R_{\lambda});B_{\lambda, i},D_{\lambda,i}(r'))_{\lambda\in I}))$$
is an isomorphism for each $1\leq i\leq M$ and each $r'>0$. It follows from the five lemma and the six-term exact sequence in $K$-theory that the map
$$e_*:\lim\limits_{d\to \infty}K_*(A_L((P_d(X^R_{\lambda});B_{\lambda},B_{\lambda}(r'))_{\lambda\in I}))\to \lim\limits_{d\to \infty}K_*(A((P_d(X^R_{\lambda});B_{\lambda},B_{\lambda}(r'))_{\lambda\in I}))$$
induced by the evaluation-at-one map is an isomorphism for each $r'$. This finshes the proof.
\end{proof}

Now, we still need to construct an index map which enables us to reduce Proposition \ref{prop:local coefficients} to Proposition \ref{prop:twisted CBC for subgroups}.  We shall follow the steps of the constructions in Section \ref{Sect: Index map} to construct the index maps
$${\rm Ind}_{F}: K_*(A((P_d(X^R_{\lambda}); B_{\lambda})_{\lambda\in I}))\to K_*(A((P_d(X^R_{\lambda});B_{\lambda}, U_{\lambda})_{\lambda\in I}))$$
and
$${\rm Ind}^L_{F}: K_*(A_L((P_d(X_{\lambda}); B_{\lambda})_{\lambda\in I}))\to K_*(A_L((P_d(X_{\lambda}); B_{\lambda},U_{\lambda})_{\lambda\in I}))$$
for each $d>0$.

Let us now sketch the construction of index maps in the following. For each $\lambda \in I$, we denote the Bott and Dirac map on $\mathcal{L}^2_{U_{\lambda}}$ by $D^{\lambda}$ and $C^{\lambda}$, respectively. Let $B^{\lambda}_{s',x}=(s')^{-1}D^{\lambda}+C^{\lambda}-V_{x}$ be Bott--Dirac map associated to $(s',x)$ for $s' \in [1, \infty)$ and $x \in U_{\lambda}$.
Define the operator
$$F^{\lambda}_{s',x}=B^{\lambda}_{s',x}(1+(B^{\lambda}_{s',x})^2)^{-1/2}
$$
on $\mathcal{L}^2_{U_{\lambda}}$. It is an unbounded and self-adjoint operator.
We define a sequence of operators, denoted by $F_{s'}=(F^{\lambda}_{s'})_{\lambda\in I, s'\in [1,\infty)}$, where for each $\lambda$, $F^{\lambda}_{s'}=F^{\lambda,U}_{s'+2{\rm dim}(U_{\lambda
}),0}: H_{d,B_{\lambda},U_{\lambda}}\to H_{d,B_{\lambda},U_{\lambda}}$ is the operator defined as in Construction (\ref{constr:spatial operator of subgroups geometry}). For each $\lambda\in I$, let $F^{\lambda}: L^2([1, \infty),H_{d,B_{\lambda},U_{\lambda}}) \to L^2([1, \infty),H_{d,B_{\lambda},U_{\lambda}})$ be the operator defined by $(F^{\lambda}u)(s'):=F^{\lambda}_{s'}u(s')$ for all $u\in L^2([1, \infty),H_{d,B_{\lambda},U_{\lambda}})$.

Following the arguments in Lemma \ref{lem: F as self-adjoint, odd in multiplier}, the operator $F$ is a self-adjoint, norm one, odd operator in
the multiplier algebra of $A((P_d(X^R_{\lambda});B_{\lambda},U_{\lambda})_{\lambda\in I})$.

Represented on $\oplus_{\lambda}L^2([1, \infty),H_{d,B_{\lambda},U_{\lambda}})$ via the amplification of the identity representation on $H_{d,B_{\lambda}}$, the twisted Roe algebra $A((P_d(X^R_{\lambda}); B_{\lambda})_{\lambda
\in I})$ is a subalgebra of the multiplier algebra of the twisted Roe algebra $A((P_d(X^R_{\lambda});B_{\lambda},U_{\lambda})_{\lambda\in I})$ by the same arguments in Lemma \ref{lem:twisted Roe subalgebra}.

Following Lemma \ref{lem:projections in multiplier}, we also have that for each projection  $p=(p^{\lambda})_{\lambda \in I} \in A((P_d(X^R_{\lambda}); B_{\lambda})_{\lambda
\in I})$, the function
$$p \mapsto (pFp)^2-p$$
is in the corner $pA((P_d(X^R_{\lambda});B_{\lambda},U_{\lambda})_{\lambda\in I})p$.
The results in the case of localization algebras can be proved using the arguments pointwise.
By the construction of index maps in $K$-theory, we can define the index maps
$${\rm Ind}_{F}: K_*(A((P_d(X^R_{\lambda}); B_{\lambda})_{\lambda\in I}))\to K_*(A((P_d(X^R_{\lambda});B_{\lambda}, U_{\lambda})_{\lambda\in I}))$$
and
$${\rm Ind}^L_{F}: K_*(A_L((P_d(X_{\lambda}); B_{\lambda})_{\lambda\in I}))\to K_*(A_L((P_d(X_{\lambda}); B_{\lambda},U_{\lambda})_{\lambda\in I}))$$
for each $d>0$.

To study the index maps, we need the twisted Roe algebras in the following. Let $A^*((P_d(X^R_{\lambda});B_{\lambda})_{\lambda\in I})$ be the $C^*$-subalgebra of $\prod_{\lambda\in I} C_b([1,\infty),C^*(P_d(X^R_{\lambda}); H_{d,B_{\lambda},U_{\lambda}}))$ generated by the paths $T=(T^{\lambda}_s)_{s\in[1,\infty), \lambda\in I}$ satisfying:
\begin{enumerate}[(1)]
    \item $\sup_{\lambda\in I, s\in [1,\infty)} {\rm Prop}_{P_d(X^R_{\lambda})}(T^{\lambda}_s)<\infty$;
    \item $\lim\limits_{s \to \infty}\sup_{\lambda\in I} {\rm Prop}_{B_{\lambda}}(T^{\lambda}_s)=0$.
\end{enumerate}
We can also define the twisted localization algebra, $A_L^*((P_d(X^R_{\lambda});B_{\lambda})_{\lambda\in I})$, to be the $C^*$-algebra of all the uniformly bounded and uniformly continuous paths $(t \mapsto T(t)=(T^{\lambda}_s(t))_{s\in[1,\infty), \lambda\in I})$ on the interval $[1,\infty)$ such that
$$
\lim\limits_{t \to \infty} \sup_{\lambda\in I, s\in [1,\infty)} {\rm Prop}_{P_d(X^R_{\lambda})}(T_s^{\lambda}(t))=0.
$$
There is a natural evaluation-at-one map
$$e: A_L^*((P_d(X^R_{\lambda});B_{\lambda})_{\lambda\in I}) \to A^*((P_d(X^R_{\lambda});B_{\lambda})_{\lambda\in I}).$$
For any fixed $s'_0$, consider the evaluation maps:
$$
\imath^{s'_0}:A((P_d(X^R_{\lambda});B_{\lambda}, U_{\lambda})_{\lambda\in I}) \to A^*((P_d(X^R_{\lambda}); B_{\lambda})_{\lambda\in I})
$$
by
$$
(T_{s,s'}^{\lambda})_{s,s'\in [1,\infty),\lambda\in I} \mapsto (T_{s,s'_0}^{\lambda})_{s \in [1,\infty),\lambda\in I}
$$
and
$$
\imath^{s'_0}_{L}:A_L((P_d(X^R_{\lambda});B_{\lambda}, U_{\lambda})_{\lambda\in I}) \to A_L^*((P_d(X^R_{\lambda});  B_{\lambda})_{\lambda\in I})
$$
by
$$
(T_{s,s'}^{\lambda}(t))_{s,s',t\in [1,\infty),\lambda\in I} \mapsto (T_{s,s'_0}^{\lambda}(t))_{s,t\in [1,\infty),\lambda\in I}.
$$
Following from the proof of Proposition \ref{prop:index map iso}, we have that the compositions $\imath^{s'_0}_*\circ {\rm Ind}_F$ and $\imath^{s'_0}_{L,*}\circ {\rm Ind}^L_F$ are isomorphisms, where $\imath^{s'_0}_*$ and $\imath^{s'_0}_{*}$ are the maps induced by the evaluation maps on $K$-theory.

Moreover, we have the following commutative diagram:
\begin{equation*}\label{diag:diagram of twisted Roe algebras}
\begin{tikzcd}
    K_*(A_{L}(P_d(X^R_{\lambda});B_{\lambda})_{\lambda\in I}))\ar[d, "{\rm Ind}^L_F"]\ar[r,"e_*"]& K_*(A(P_d(X^R_{\lambda});B_{\lambda})_{\lambda\in I}))\ar[d,"{\rm Ind}_F"]\\
    K_*(A_L(P_d(G_m); B_{\lambda},U_{\lambda})_{\lambda\in I}))\ar[r,"e_*"]\ar[d,"\imath^{s'_0}_{L,*}"]& K_*(A(P_d(G_m); B_{\lambda},U_{\lambda})_{\lambda\in I}))\ar[d, "\imath^{s'_0}_*"]\\
    K_*(A^*_{L}(P_d(X^R_{\lambda});B_{\lambda})_{\lambda\in I}))\ar[r,"e_*"]&K_*(A^*(P_d(X^R_{\lambda});B_{\lambda})_{\lambda\in I})).
\end{tikzcd}
\end{equation*}
Since the horizontal map in the middle and the vertical composition maps are isomorphisms when $d$ passing to infinity, Proposition \ref{prop:local coefficients} follows from the diagram chasing.

\section{Proof of the main theorem}

In this section, we shall prove Theorem \ref{theorem-CBC-for-sequence}, and the main result Theorem \ref{thm:main result} follows from Theorem \ref{theorem-CBC-for-sequence}.
\begin{proof}[Proof of \ref{theorem-CBC-for-sequence}.]
For a fixed $s>0$, we have the following commutative diagram:
\begin{equation*}
\begin{tikzcd}
    \lim\limits_{d \to \infty}K_*(C^*_{u,L}(P_d(G_m);H_{d,m})_{m \in \mathbb{N}}))\ar[d, "{\rm Ind}^L_F"]\ar[r,"e^1_*"]& \lim\limits_{d \to \infty}K_*(C^*_{u}(P_d(G_m);H_{d, m})_m))\ar[d,"{\rm Ind}_F"]\\
    \lim\limits_{d \to \infty}K_*(A_L(P_d(G_m); V_m)_{m \in \mathbb{N}}))\ar[r,"e^2_*"]\ar[d,"\imath^{s}_{L,*}"]& \lim\limits_{d \to \infty}K_*(A(P_d(G_m); V_m)_{m \in \mathbb{N}}))\ar[d, "\imath^s_*"]\\
    \lim\limits_{d \to \infty}K_*(C_{u,L}((P_d(G_m);H_{d,V_m})_{m \in \mathbb{N}}))\ar[r,"e^3_*"]& \lim\limits_{d \to \infty}K_*(C_{L}((P_d(G_m);H_{d,V_m})_{m \in \mathbb{N}})),
\end{tikzcd}
\end{equation*}
where we use $e^1_*$, $e^2_*$ and $e^3_*$ to distinguish the maps induced by the evaluation-at-one maps on $K$-theory.
By Proposition \ref{prop:index map iso}, we know that the vertical compositions are isomorphisms. By Proposition \ref{prop:twistedCBC}, we know that
the map
$$e^2_*:\lim\limits_{d \to \infty}K_*(A_L(P_d(G_m); V_m)_{m \in \mathbb{N}}))\to \lim\limits_{d \to \infty}K_*(A(P_d(G_m); V_m)_{m \in \mathbb{N}}))
$$
is an isomorphism.

Now, let us prove that the map
$$
 e^1_*: \lim\limits_{d \to \infty}K_*(C^*_{u,L}(P_d(G_m);H_{d,m})_{m \in \mathbb{N}}))\to \lim\limits_{d \to \infty}K_*(C^*_{u}(P_d(G_m);H_{d, m})_{m\in \mathbb{N}}))
$$
is an isomorphism.

Let $x\in \ker(e^1_*)$. Then $e^2_* \circ {\rm Ind}_F^{L} (x)={\rm Ind}_F \circ e^1_* (x)=0$. By the injectivity of $e^2_*$ in the middle, we have that ${\rm Ind}_F^{L} (x)=0$ which implies that $\imath^s_{L,*}\circ{\rm Ind}_F^L(x) =0$. In addition, $\imath^s_{L,*}\circ{\rm Ind}_F^L$ is an isomorphism, therefore $x=0$. Thus, $e^1_*$ is injective.

Let $y\in K_*(C^*_{u}(P_d(G_m);H_{d, m})_{m\in \mathbb{N}}))$. Since the map $e^2_*$ in the middle is an isomorphism, there exists $z \in \lim\limits_{d \to \infty} K_*(A_L(P_d(G_m); V_m)_{m \in \mathbb{N}}))$ such that $e^2_*(z)={\rm Ind}_F(y)$. Since the composition $\imath^s_{L,*}\circ {\rm Ind}_F^L$ is an isomorphism, there exists a unique $z'\in K_*(C^*_{u,L}(P_d(G_m);H_{d, m})_{m\in \mathbb{N}}))$ such that
$$
\imath^s_{L,*}\circ {\rm Ind}^L_F(z')=\imath^s_{L,*}(z).
$$
By the commutativity of the diagram and the injectivity of the composition $\imath^s_{*}\circ {\rm Ind}_F$, we know that $e^1_*(z')=y$. Thus, $e^1_*$ is surjective. This finishes the proof.
\end{proof}
\subsection*{Acknowledgement} The authors would like to thank Rufus Willett for helpful communications. The authors also thank the anonymous referees
whose comments and suggestions helped to improve the paper. The first author would like to thank Professor Matthew Kennedy for his support.


%

\vskip 1.5cm

\begin{itemize}
\item[] Jintao Deng \\
Department of Mathematics, University of Waterloo, Waterloo, ON, Canada. \\
E-mail: jintao.deng@uwaterloo.ca

\item[] Qin Wang \\
Research Center for Operator Algebras,  and Shanghai Key Laboratory of Pure Mathematics and Mathematical Practice, School of Mathematical Sciences, East China Normal University, Shanghai, 200241, P. R. China.\\
E-mail: qwang@math.ecnu.edu.cn

\item[] Guoliang Yu \\
Department of Mathematics, Texas A\&M University, College Station, TX 77843-3368, USA.\\
E-mail: guoliangyu@math.tamu.edu

\end{itemize}

\end{document}